\documentclass[]{article}
\usepackage[utf8]{inputenc}
\usepackage[english]{babel}
\usepackage[T1]{fontenc}

\usepackage{amssymb,amsmath,amscd,amsthm}
\usepackage[pdftex]{graphicx}
\usepackage{url}
\usepackage{algorithm}
\usepackage[noend]{algpseudocode}
\usepackage[numbers]{natbib}
\usepackage{graphicx}
\usepackage{listings}
\newcommand{\R}{\mathbb{R}}
\newcommand{\Z}{\mathbb{Z}}
\newcommand{\X}{\mathbb{X}}
\newcommand{\pd}{\mathrm{PD}}
\newcommand{\bdsp}{D}

\newcommand{\subg}[2]{G_X(\succ_{#1} #2)}

\newcommand{\isubg}[2]{G_X(\succeq_{#1} #2)}
\newcommand{\isubgv}[2]{V_X(\succeq_{#1} #2)}
\newcommand{\isubge}[2]{E_X(\succeq_{#1} #2)}
\newcommand{\OV}{\mathrm{OV}}
\newcommand{\SV}{\mathrm{SV}}
\newcommand{\SVo}{\tilde{\mathrm{SV}}}
\newcommand{\Kv}{K_{\mathrm{V}}}
\newcommand{\Ke}{K_{\mathrm{E}}}
\newcommand{\pt}{\bar{G}}

\newtheorem{theorem}{Theorem}
\newtheorem{definition}[theorem]{Definition}
\newtheorem{prop}[theorem]{Proposition}
\newtheorem{lemma}[theorem]{Lemma}
\newtheorem{fact}[theorem]{Fact}
\newtheorem{cond}[theorem]{Condition}
\newtheorem{claim}[theorem]{Claim}
\newtheorem{remark}[theorem]{Remark}

\newtheorem{theoremext}{Theorem}

\newtheorem{propext}[theoremext]{Proposition}

\title{Stable Volumes for Persistent Homology}
\author{Ippei Obayashi}

\begin{document}

\maketitle

\begin{abstract}
  This paper proposes a stable volume and a stable volume variant, referred to as a stable sub-volume, for more reliable data analysis using persistent homology. 
  In prior research, an optimal cycle and similar ideas have been proposed to identify the homological structure
  corresponding to each birth-death pair in a persistence diagram.
  While this is helpful for data analysis using persistent homology,
  the results are sensitive to noise. In this paper, stable volumes and stable sub-volumes are proposed to solve this problem.
  For a special case, we prove that a stable volume is the robust part of an optimal volume against noise.
  We implemented stable volumes and sub-volumes on HomCloud, a data analysis software package based on persistent homology, and show
examples of stable volumes and sub-volumes.
\end{abstract}

\section{Introduction}\label{sec:intro}

Topological data analysis (TDA)\cite{eh,carlsson} is a data analysis method utilizing the mathematical concept of
topology. In recent years, persistent homology (PH)\cite{elz,zc} has become one of the most important tools of TDA.
PH is mathematically formalized by using homology on a filtration. We can characterize the
geometric information by encoding the information regarding the scale of the data on the filtration.
PH has developed rapidly over the most recent decade and has been applied in a variety of areas, including natural image analysis\cite{Carlsson2008}, biology \cite{virus}, geology \cite{Suzuki2021}, 
and materials science~\cite{Hiraoka28062016,granular,YoheiONODERA201919143,Hirata2020}.

PH information can be described by a \emph{persistence diagram} (PD) or a
persistence barcode\footnote{ A persistence diagram and a persistence barcode contain the same information. The difference is how the information is visualized.}.
A PD is a scatter plot on the XY plane, where each point (called a birth-death pair) on the plot
corresponds to a homological structure in the input data.

\subsection{Persistent homology and volume-optimal cycles}

Figure \ref{fig:ph5-1} can be used to intuitively explain how a PD is determined from data.
The input data in the example constitute a pointcloud with five points, as shown in Fig. \ref{fig:ph5-1}(a). The points appear as
a regular triangle and a square. Let $a$ be the length of edges of the regular triangle and the square.
Since the five points themselves do not carry any interesting topological information,
we construct a topological structure by putting circles with values of radii $r$ as indicated in Fig. \ref{fig:ph5-1}(b)(c)(d)(e).

\begin{figure}[hbtp]
  \centering
  \includegraphics[width=\hsize]{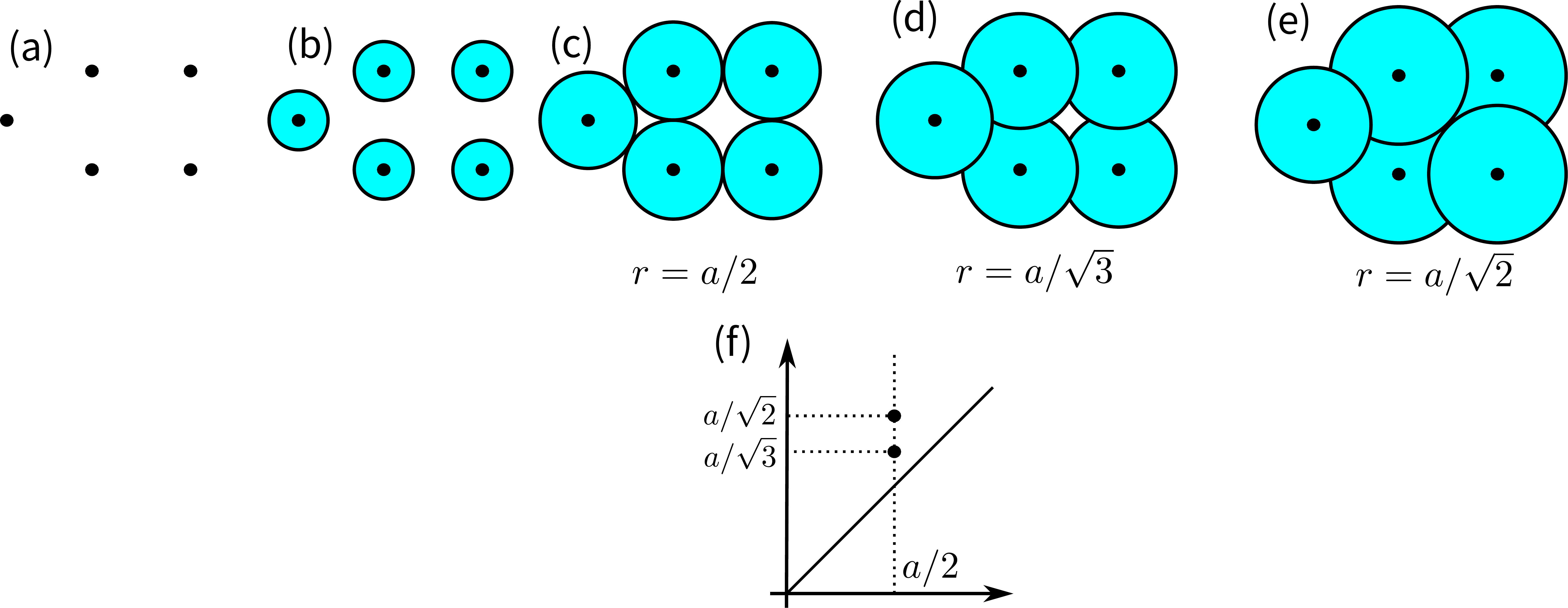}
  \caption{How to compute a Persistence diagram for 5 points
    (a) Input pointcloud (b) Pointcloud and circles with radii $r < a/2$ (c) $r = a/2$ (d) $r = a / \sqrt{3}$ (e) $r = a / \sqrt{2}$ (f) Persistence diagram }
  \label{fig:ph5-1}
\end{figure}

We now face the problem of how to determine the proper size of the circles. If the circles are too small, as in Fig.~\ref{fig:ph5-1}(b), the topology is simply equivalent to the five points. On the other hand, if the circles are too large, as in Fig.~\ref{fig:ph5-1}(e), the shape becomes acyclic, which makes it impossible to uncover any interesting topological information. PH solves this problem by considering the appearance and disappearance of homology generators associated with radii changes;
that is, PH considers the increasing sequence from (b) to (e).

In the Fig.~\ref{fig:ph5-1} example, two holes (homology generators in the 1st homology) appear at (c),
one of which disappears at (d). The other hole disappears at (e).
The pair of radii\footnote{The squares of radii are also often used in the literature of PH.}
of the appearance and disappearance of each homology generator is called a \emph{birth-death} pair,
and the multiset of all birth-death pairs is called a \emph{persistence diagram}.
In this example, the PD for the 1st homology is $\{(a/2, a/\sqrt{3}), (a/2, a/\sqrt{2})\}$. The persistence diagram
is visualized by a scatter plot or a 2D histogram (Fig.~\ref{fig:ph5-1}(f)).
The two pairs correspond to a regular triangle and a square. We can extract the geometric information of
the pointcloud in Fig.~\ref{fig:ph5-1}(a) using the PD.

It is beneficial in data analysis using PH to detect the original homological structures corresponding to each birth-death pair. This is sometimes called ``inverse analysis on PH''. Some applications of PH\cite{Hiraoka28062016, Hirata2020} already use 
inverse analysis on PH.
However, detection is not an easy task, as the representative cycle corresponding to a homology generator is not unique.

To solve this problem, various methods involving the solution of an optimization problem in homology algebra has been proposed\cite{Chen2011,Escolar2016,p-1-cycle,Chen2011,Erickson2005,optimal-Day,voc,sensor-l0-l1,Ben-thesis} in various settings.
For PH, optimal cycles\cite{Escolar2016}, volume-optimal cycles\cite{voc}, and persistent 1-cycles\cite{p-1-cycle} have been proposed. These studies give formalizations suitable to PH.
The ``tightest'' or ``minimal'' cycle is considered the best to interpret, and solving the optimization problem in homology algebra produces the tightest cycle.  As an example, consider the simplicial complex in Fig.~\ref{fig:tightest5}. The 1st homology group of the simplicial complex $H_1$ is isomorphic to $\Z$.
$z_1$ and $z_2$ give the same generator of $H_1$, but $z_1$ is superior since the $z_1$ surrounds the
hole more tightly than $z_2$. The above methods find $z_1$ by using the prescribed optimization technique.

It should be noted that there are two choices of functions to be optimized. One is the size of the cycle; the other is the internal volume of the cycle. In Fig.~\ref{fig:tightest5}, $z_1$ is tighter than $z_2$ in both senses; however, it is possible for the solutions to the two optimization problems to differ. The better choice depends on the problem. This paper mainly uses the internal volume as an optimization function since it is suitable to define stable volumes in Section~\ref{sec:stable-volume-n-1}.

\begin{figure}[hbpt]
  \centering
  \includegraphics[width=0.3\hsize]{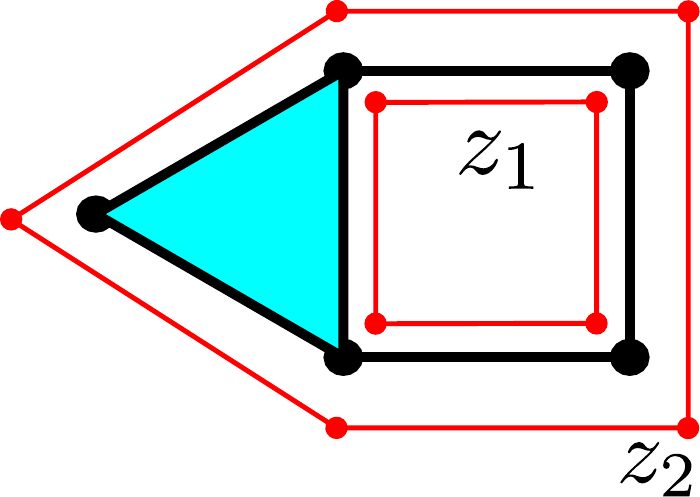}
  \caption{A simplicial complex with one hole}
  \label{fig:tightest5}
\end{figure}

\subsection{Problems of homology optimization}

The existing methods have the following two problems:
\begin{itemize}
\item The methods do not always give the minimal building blocks of the data
\item The result is unstable to noise
\end{itemize}
The example in Fig.~\ref{fig:ph5-1} can be used to demonstrate these problems.

The pointcloud in Fig.~\ref{fig:ph5-1} has two minimal building blocks, a triangle and a square.
However, the optimization technique sometimes fails to find these minimal building blocks when a small noise is added.
Figure~\ref{fig:ph5-2}(a) shows the pointcloud in Fig.~\ref{fig:ph5-1}(a) when a small noise is added.
This figure also shows circles, and, as can be seen, a pentagon appears before either a triangle or a square appears.
By applying the homology optimization technique to the data in Fig.~\ref{fig:ph5-2}(a), we produce a triangle and
a pentagon, as shown in Fig.~\ref{fig:ph5-2}(b). On the other hand, consider the pointcloud in Fig.~\ref{fig:ph5-2}(c).
The pointcloud here is also very close to that in Fig.~\ref{fig:ph5-1}(a); however, while a square appears first, the homology optimization technique gives a square and a triangle, as shown in Fig.~\ref{fig:ph5-2}(d).
\begin{figure}[hbpt]
  \centering
  \includegraphics[width=0.8\hsize]{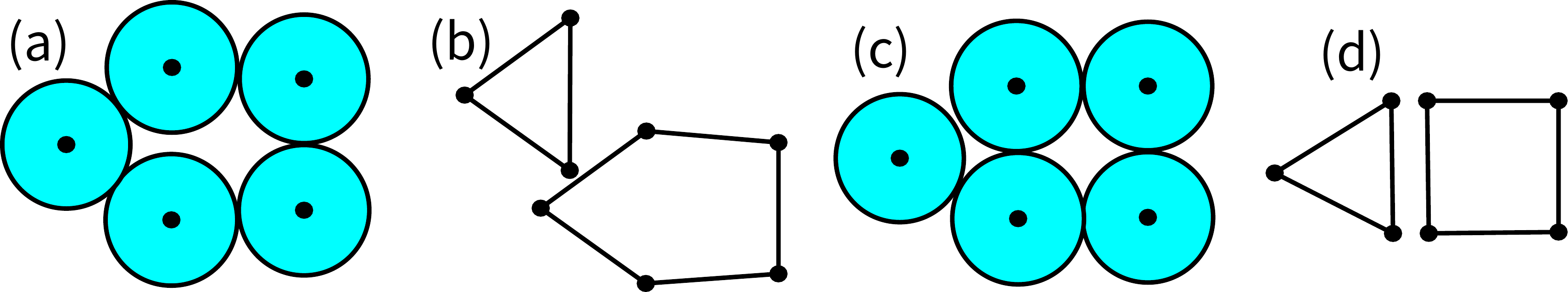}
  \caption{Two types of optimization results}
  \label{fig:ph5-2}
\end{figure}

The example demonstrates the following problems: (1) A small noise can change the result, and (2) The optimization technique sometimes fails to give minimal building blocks of the data. The first problem is related to the reliability of the analysis; the second problem is related to the interpretability of the result.

Previous studies~\cite{Cohen-Steiner2007, Chazal2009, Bauer2014, Lesnick2015} have proved the stability theorem for persistence diagrams.
The literature indicates that a PD is continuously changed by a small noise in the input data if we consider reasonable metrics.
The stability of other PH outputs has also been studied in the context of machine learning and PH\cite{JMLR:v16:bubenik15a, JMLR:v18:17-317, JMLR:v18:16-337} or pointcloud summary method\cite{kurlin2015one,smith2021skeletonisation}.
Although these types of stability  play an important role in the study of PH, such a stability theorem
does not hold for the solution of homology optimization problems.

Both in Fig.~\ref{fig:ph5-2}(a) and Fig.~\ref{fig:ph5-2}(c), the minimal building blocks are a triangle and a square, and it is more natural that an inverse analysis technique gives a triangle and a square for both data. Therefore, we would hope to detect such structures even if a small noise is added.

These problems are practically significant, especially when we analyze crystalline structures using PH. We can demonstrate the problems using the synthetic crystalline data in Fig.~\ref{fig:crystalline}.
Here, the pointcloud consists of 27 points, arranged in a 3x3x3 cubical lattice (Fig.~\ref{fig:crystalline}(a)). The distance between two vertices on the cube is $1$. A small noise is added to the pointcloud. Figure~\ref{fig:crystalline}(b) shows the 1st PD of the pointcloud. As can be seen here, some birth-death pairs are concentrated around $(1/4, 1/\sqrt{2}) \approx (0.5, 0.7)$. These pairs correspond to 1x1 squares in the lattice. By applying volume-optimal cycles\cite{voc}, we found some loops corresponding to the pairs in Fig.~\ref{fig:crystalline}(c) and (d). Some of the cycles shown in (c) are squares, as expected; however, others in (d) are larger structures resembling chairs. These larger structures were detected with the same mechanism in Fig.~\ref{fig:ph5-2}.

\begin{figure}[hbt]
  \centering
  \includegraphics[width=0.9\hsize]{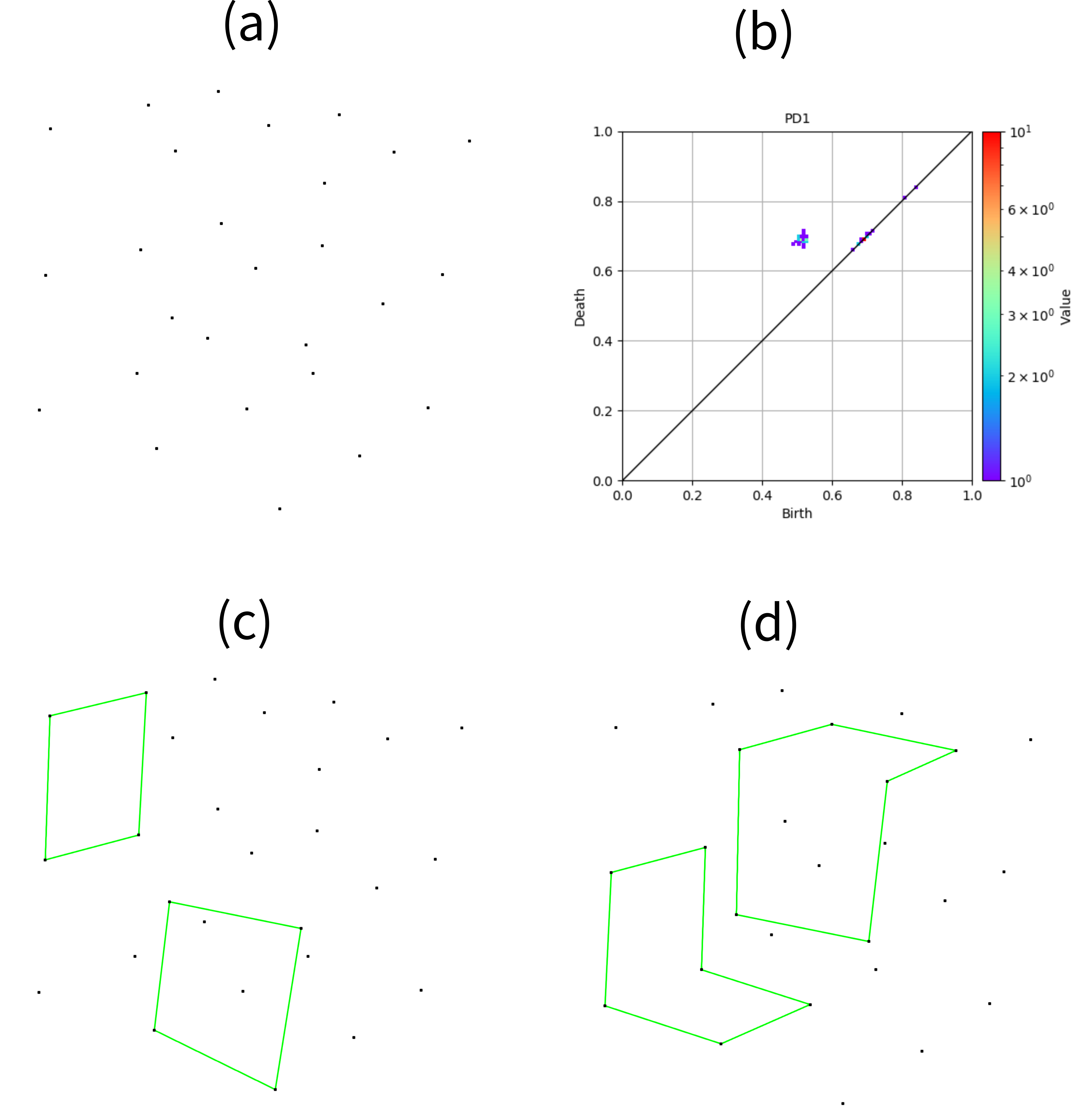}
  \caption{Optimal volumes for 3x3x3 cubical lattice. (a) The pointcloud, (b) 1st PD, (c) Optimal volumes of two pairs, (d) Optimal volumes of other two pairs}
  \label{fig:crystalline}
\end{figure}

The purpose of this paper is to present a new method to solve these problems. That is, we propose
a method to find minimal building blocks that is robust to noise.

\subsection{Statistical approach in previous research}\label{subsec:intro-stat-approach}
Bendich, Bubenik, and Wagner~\cite{stabilization} proposed a statistical approach for the problem,
which can be outlined as follows:
\begin{enumerate}
\item Computing a PD from the input data and choosing a birth-death pair to analyze
\item Adding a small noise to the input data, computing a PD, and applying inverse analysis  
\item Repeating (2) multiple times
\item Computing an average of the results of the inverse analysis
\end{enumerate}
We can flexibly change the method of inverse analysis and the way in which the average of the results is computed.
We can combine the statistical approach with optimal or volume-optimal cycles are as follows:

\begin{algorithm}
  \caption{Computing persistence trees by the merge-tree algorithm.}\label{alg:volopt-hd-compute}
  \label{alg:stat}
  \begin{enumerate}
  \item Computing a PD from the input pointcloud and choosing a birth-death pair to analyze
  \item Adding a small noise to the input data and computing the PD and optimal cycles or volume-optimal cycles
    of the corresponding birth-death pair
  \item Repeating (2) multiple times
  \item For each point in the pointcloud, computing the frequency of $(\text{the point}) \in (\text{optimal cycle})$
  \end{enumerate}
\end{algorithm}

By applying this method to the birth-death pair $(1/2, 1/\sqrt{2})$ in Fig.~\ref{fig:ph5-1}, we produce the result shown in Fig.~\ref{fig:ph5-3}. The result indicates that the four points on the square are robust to noise and that the leftmost point is less robust. This result is consistent with the fact that the pair $(1/2, 1/\sqrt{2})$ corresponds to a square.

\begin{figure}[hbpt]
  \centering
  \includegraphics[width=0.2\hsize]{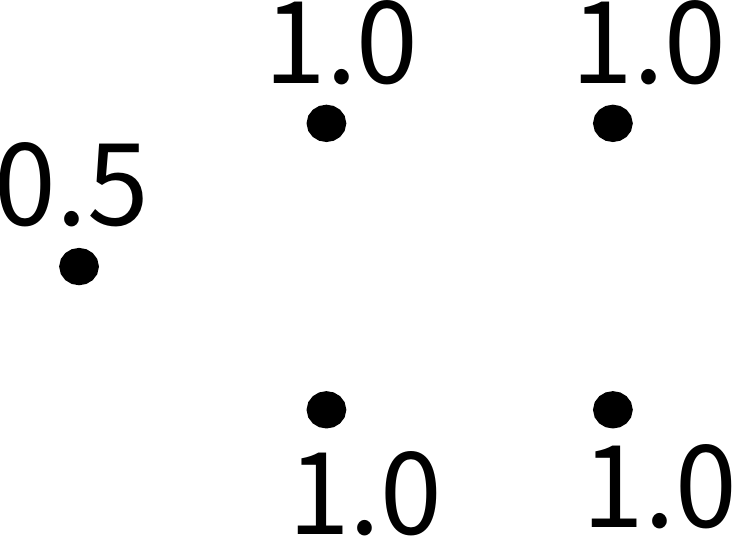}
  \caption{Result of applying the statistical approach to the birth-death pair $(1/2, 1/\sqrt{2})$}
  \label{fig:ph5-3}
\end{figure}

The stability of the method in a probabilistic sense was also proved by Bendich, Bubenik, and Wagner. Notably, the method provides a more reliable inverse analysis.
The idea is quite clever, simple, and easy to implement. However,
the method has a high computation cost since it requires the user to compute PDs and optimal cycles a large number of times. In \cite{stabilization}, the authors repeat the computation of the generators 1,000 or 10,000 times.
We suspect that fewer repetitions may be sufficient, but ten or a hundred trials will be necessary.
Multiple trials and errors are typically needed to tune the noise bandwidth in order to apply the method, and thus the cost is not ignorable.

\subsection{Reconstructed shortest cycles in Ripserer.jl}
The ``reconstructed shortest cycles'' functionality of Ripserer.jl\cite{cufar2020ripserer}\footnote{\url{https://github.com/mtsch/Ripserer.jl}} by Čufar
gives another solution of the problem.
The functionality reconstructs the tightest 1-cycle using the shortest path algorithm and the representative of persistent \emph{cohomology}.
The functionality accepts a noise bandwidth parameter and computes a tighter loop.

The advantage of reconstructed shortest cycles is its efficiency.
Since it uses the shortest path algorithm, the computational complexity is small.
However, reconstructed shortest cycles can be applied only to 1st PH since it uses the shortest path algorithm.
Another disadvantage is the lack of mathematical justification of the functionality.
Now\footnote{Dec. 4, 2021} the functionality is declared as experimental and gives no mathematical documentation.
We explain the algorithm in Appendix~\ref{sec:rsc}.

\subsection{Results}

In this paper, we propose stable volumes and a variant of stable volumes, called stable sub-volumes.
The proposed method is based on the volume-optimal cycles and optimal volumes shown in \cite{voc}.
Stable volumes produce minimal building blocks with a lower computation cost than the statistical approach. 
Below, we provide an outline of stable volumes.
The exact formalization will be introduced in Section~\ref{sec:preliminaries} and further developed in subsequent sections. 

Let $X$ be a simplicial complex and $X^{(k)}$ the set of all $k$-simplices of $X$. We define a \emph{level function} as follows.
\begin{definition}
  $\hat{r} : X \to \R$ is a \emph{level function} if $\sigma \subsetneq \sigma'$ implies $\hat{r}(\sigma) \leq \hat{r}(\sigma')$.
\end{definition}

Here, for simplicity, we assume that a level function $\hat{r}$ can distinguish every simplex; that is, we assume the following:
\begin{equation}\label{eq:distinguishability}
 \sigma \not = \sigma' \text{ implies } \hat{r}(\sigma) \not = \hat{r}(\sigma'). 
\end{equation}

From the definition of the level function, $X_{\hat{r}, t} = \{\sigma \mid \hat{r}(\sigma) < t \}$ is a subcomplex of $X$ and $t < t'$ implies $X_{\hat{r}, t} \subseteq X_{\hat{r}, t'}$. This means that $\{X_{\hat{r}, t}\}_{t \in \R}$ is a filtration of simplicial complexes. From the filtration, the $k$th PD, $\pd_k(\{X_{\hat{r}, t}\})$, is defined. Any birth-death pair in the diagram can be written as $(b, d) = (\hat{r}(\tau_b), \hat{r}(\omega_d))$, where $\tau_b$ is a $k$-simplex and $\omega_d$ is a $(k+1)$-simplex due to the theory of PH.

The \emph{optimal volume} of the pair $(b, d)$ is defined by the solution to the following minimization problem:
\begin{equation}
  \begin{aligned}
    &\text{minimize } \|z\|_0, \text{ s. t. } \\
    z &= \omega_d + \sum_{\omega \in \mathcal{F}_{k+1}} \alpha_\omega \omega \in C_{k+1}(X; \Bbbk), \\
    \tau^*(\partial z) &= 0 \text{ for every } \tau \in \mathcal{F}_k, \\
    \tau_b^*(\partial z) & \not = 0,
  \end{aligned}
\end{equation}
where $C_{k+1}(X; \Bbbk)$ is a $(k+1)$th chain complex whose coefficient field is $\Bbbk$, $\mathcal{F}_k = \{ \sigma \in X^{(k)} \mid \hat{r}(\tau_d) < \hat{r}(\sigma) < \hat{r}(\sigma_d) \}$, $\tau^*$ is an element of the dual basis of the cochain complex $C^k(X; \Bbbk)$, and $\|z\|_0 = \{\omega \mid \alpha_\omega \not = 0 \}$ is the $\ell^0$ norm of $z$\footnote{The $\ell^0$ norm is, in fact, a norm since the triangle inequality does not hold. However, this is often called the $\ell^0$ norm in the context of machine learning and sparse modeling.}. The optimal volume minimizes the internal volume surrounded by the loop. The optimal volume is defined on $\Bbbk = \Z / 2\Z$; however, in practice, optimal volumes are determined by using linear programming with a change of coefficient field from $\Z / 2\Z$ to $\R$ and an approximation of the optimization of the $\ell^0$ norm by the $\ell^1$ norm. For optimal volume $z$, $\partial z$ is called a \emph{volume-optimal cycle}. Reference~\cite{voc} shows that optimal volumes and volume-optimal cycles have good mathematical properties and provide good representations of a birth-death pair on a PD.

The \emph{stable volume} for the pair $(b, d)$ with a noise bandwidth parameter $\epsilon$ is defined by the solution to the following minimization problem:
\begin{equation}
  \begin{aligned}
    &\text{minimize } \|z\|_0, \text{ s. t. } \\
    z &= \omega_d + \sum_{\omega \in \mathcal{F}_{\epsilon, k+1}} \alpha_\omega \sigma \in C_{k+1}(X; \Bbbk), \\
    \tau^*(\partial z) &= 0 \text{ for every } \tau \in \mathcal{F}_{\epsilon, k}, \\
  \end{aligned}
\end{equation}
where $\mathcal{F}_{\epsilon, k} = \{ \sigma \in X^{(k)} \mid \hat{r}(\tau_d) + \epsilon \leq \hat{r}(\sigma) < \hat{r}(\omega_d)\}$. This is quite similar to the definition of optimal volumes. The following theorem states that the stable volume is considered to be the ``robust part'' against noise.

\begin{theorem}\label{thm:optimal-stable-codim1}
  Let $X$ be an $n$-dimensional simplicial complex, and let $\hat{r}$ be a level function. Assume that $X$ is embedded in $\R^n$. We consider $\Bbbk = \Z/2\Z$ as a coefficient field. For any $(b, d) \in \pd_{n-1}(\{X_{\hat{r}, t}\}_{t \in \R})$, $\OV(\hat{r}, b, d)$ and $\SV_{\epsilon}(\hat{r}, b, d)$  are defined as follows:
  \begin{equation}
    \begin{aligned}
      \OV(\hat{r}, b, d)  = & \{\omega \in X^{(n)} \mid \omega^*(z) \not = 0, \text{ $z$ is an optimal volume of $(b, d)$} \}, \\
      \SV_{\epsilon}(\hat{r}, b, d) =& \{\omega \in X^{(n)} \mid \omega^*(z) \not = 0, \ z \text{ is a stable volume of } (b, d) \\
      & \text{ with noise bandwidth parameter } \epsilon  \}.
    \end{aligned}
  \end{equation}
  Then the following holds for sufficiently small $\epsilon > 0$.
  \begin{equation}
    \begin{aligned}
      \SV_\epsilon(\hat{r}, b, d) = \cap_{\hat{q} \in \mathcal{R}_{\epsilon}} \OV(\hat{q}, b_{\hat{q}, \omega_d}, \hat{q}(\omega_d)),
    \end{aligned}
  \end{equation}
  where $\omega_d$, $\mathcal{R}_\epsilon$ and $b_{\hat{q}, \omega_d}$ are defined as follows:
  \begin{align}
    d &= \hat{r}(\omega_d), \\
      \mathcal{R}_\epsilon & = \{\hat{q}:
      \text{a level function satisfying the condition
        \eqref{eq:distinguishability}} \mid \nonumber \\
      & \ \  \|\hat{q} - \hat{r}\|_{\infty} < \epsilon/2,\text{ and } \nonumber \\
      & \ \ \hat{r}(\sigma) < \hat{r}(\omega_d) \text{ implies } \hat{q}(\sigma) < \hat{q}(\omega_d) \nonumber \\
      & \},  \\
    b_{\hat{q}, \omega_d} &: \text{ a birth time paired with } \hat{q}(\sigma_d) \text{ in } PD_{n-1}(\{X_{\hat{q},t}\}_{t \in \R}\}.
  \end{align}
\end{theorem}

We can regard $\mathcal{R}_\epsilon$ as the set of all possible level functions with small noise $\epsilon/2$. Then the theorem states that the stable volume is the common area of all possible optimal volumes under small noise. This theorem does not hold for the PH of another dimension, but it suggests that stable volumes give a better result than optimal volumes.

For practical applications, assumption \eqref{eq:distinguishability} is overly strict. Vietoris-Rips, \v{C}ech, and alpha filtrations do not satisfy the condition. Therefore, we introduce an order with levels in Section~\ref{subsec:order-with-level} and formalize all of them using the order with levels.

As the stable volume method has already been implemented in HomCloud~\cite{homcloud},  stable volumes can now be used to analyze the data.

We can demonstrate stable volumes by using the same input data as in Fig.~\ref{fig:crystalline}. Figure~\ref{fig:crystalline-2} shows the
optimal volumes (Fig.~\ref{fig:crystalline-2}(a)) and stable volumes (Fig.~\ref{fig:crystalline-2}(b)) of the same two birth-death pairs. The birth-death pairs correspond to 1x1 squares in the lattice points, but Fig.~\ref{fig:crystalline-2}(a) shows larger loops than expected. In contrast Fig.~\ref{fig:crystalline-2}(b) shows the expected 1x1 squares. This example suggests that stable volumes give better results and help promote a better intuitive understanding of PDs. 

\begin{figure}[hbpt]
  \centering
  \includegraphics[width=0.7\hsize]{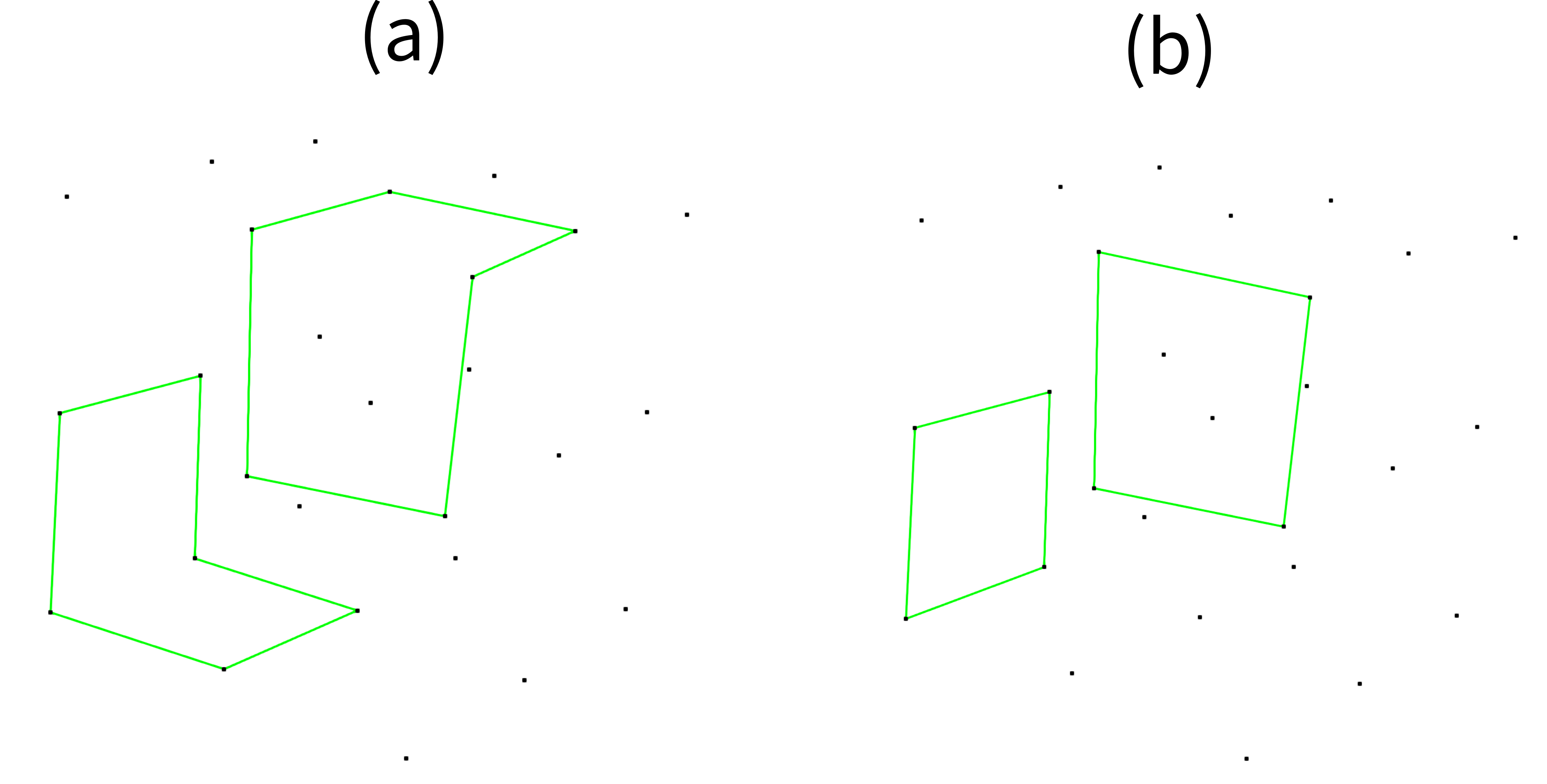}
  \caption{Stable volumes computed for 3x3x3 cubical lattice}
  \label{fig:crystalline-2}
\end{figure}

Figure~\ref{fig:lattice-2d} shows another example of stable volume. Here, we consider 2D lattice points with defects (Fig.~\ref{fig:lattice-2d} (a)). The configuration of the points is somewhat perturbed from the complete lattice, and some points are removed randomly. Figure~\ref{fig:lattice-2d}(b) shows the 1st PD of the points, and Fig.~\ref{fig:lattice-2d}(c) shows the optimal volume of (0.496, 4.371). Figure~\ref{fig:lattice-2d}(d) shows the stable volume of the same birth-death pair. The optimal volume in Fig.~\ref{fig:lattice-2d}(d) seems to surround the holes in the pointcloud more naturally than is the case in Fig.~\ref{fig:lattice-2d}(c).

\begin{figure}[htbp]
  \centering
  \includegraphics[width=0.95\hsize]{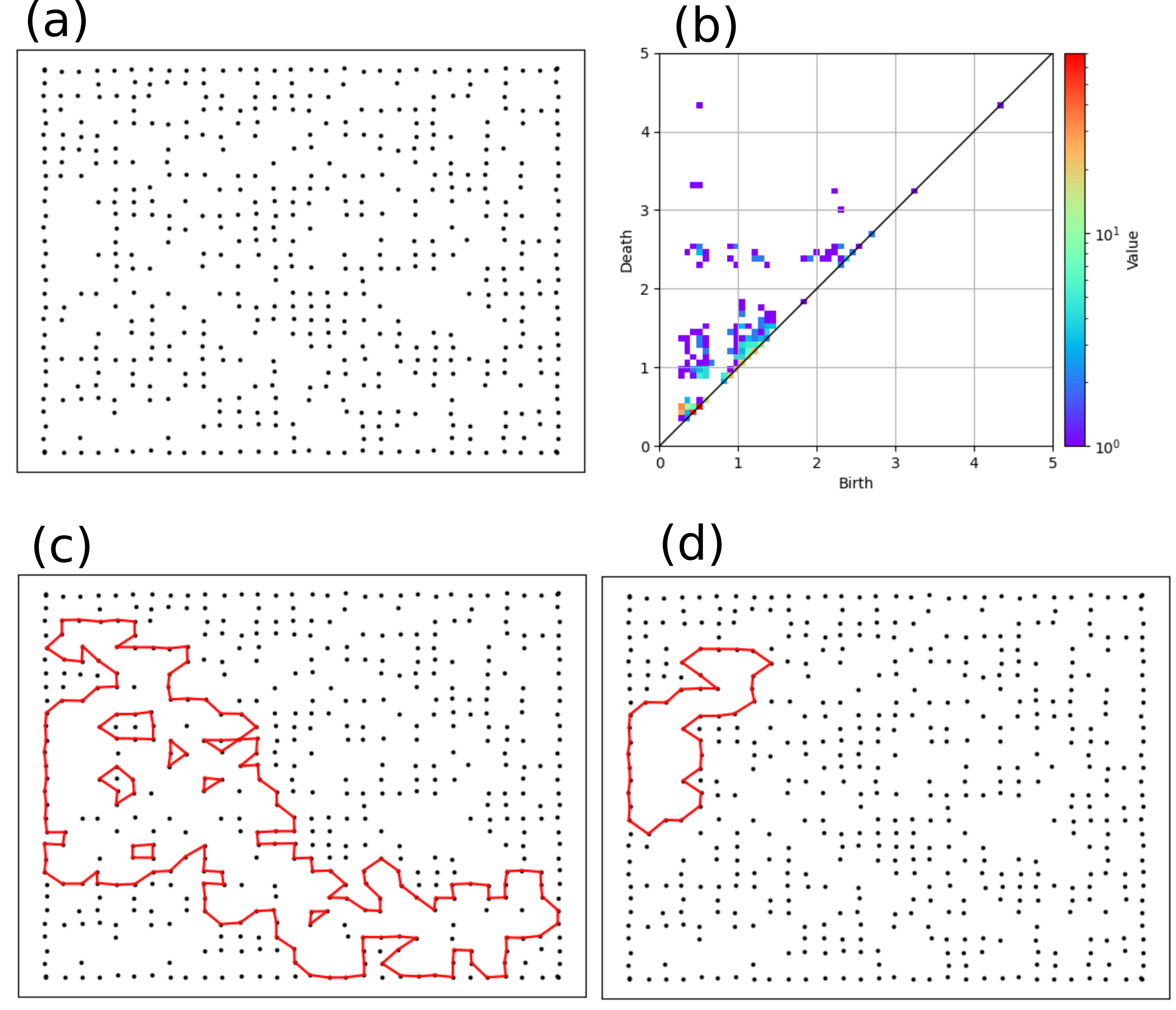}
  \caption{An optimal volume and a stable volume for a 2D lattice with defects. (a) The input pointcloud, (b) 1st PD, (c) The optimal volumes of (0.496, 4.371), (d) The stable volume of the same birth-death pair}
  \label{fig:lattice-2d}
\end{figure}

Additional demonstrations of stable volumes using synthetic and real data are given in Section~\ref{sec:example}.






\subsection{Organization of the paper}

The remainder of the paper is organized as follows:
Section \ref{sec:preliminaries} introduces the concept of PH, optimal volumes, and volume-optimal cycles.
Section \ref{sec:stable-volume-n-1} defines stable volumes for the $(n-1)$th PH embedded in $\R^n$.
In this section we also discuss the limitation of stable volumes.
Section \ref{sec:stable-volume-optimization} shows an alternative formalization of stable volumes using mathematical optimization.
Section~\ref{subsec:stable-volume-optimization-n-1} provides a proof that the two formalizations are consistent for the $(n-1)$th PH when the simplicial complex is embedded in $\R^n$.
Section~\ref{subsec:stable-volume-optimization-general} extends the formalization to other cases.
Section \ref{sec:implementation} discusses the software implementation of stable volumes. 
Section~\ref{sec:example} gives examples using synthetic and real data.
This section also provies the comparison with previous methods, the statistical method and the reconstructed shortest cycles.
Section~\ref{sec:parameter-tuning} discusses the way in which the noise bandwidth parameter can be tuned.
Section~\ref{sec:comparison} compares stable volumes and sub-volumes.
Section \ref{sec:conclusion} summarizes the paper and offers concluding remarks. In this section we also discuss the differences between the proposed method and the methods previous works.

\section{Persistent homology, optimal volumes, and volume-optimal cycles}\label{sec:preliminaries}

In this section,
we introduce the mathematical concepts of 
PH, persistence diagrams, optimal volumes,
and volume-optimal cycles, which provide the foundation for our discussion.

\subsection{Persistent homology}

PH is defined on a filtration of topological spaces. Let $X$ be a finite simplicial complex\footnote{To simplify the explanation in this paper, we use a simplicial complex. However, the content of the paper can be easily extended to cubical or cell complexes.} and
\begin{equation}
  \X: \emptyset = X_0 \subsetneq X_1 \subsetneq \cdots \subsetneq X_N = X
\end{equation}
be a filtration of subcomplexes of $X$. The $k$th \emph{persistent homology} $H_k(\X; \Bbbk)$
with a coefficient field $\Bbbk$ is defined as 
\begin{equation}
  H_k(\X; \Bbbk): H_k(X_0; \Bbbk) \to H_k(X_1; \Bbbk) \to \cdots \to H_k(X_N; \Bbbk),
\end{equation}
where $H_k(X_i; \Bbbk)$ is the $k$th homology $\Bbbk$-vector space of $X_i$ and
$\to$ is the linear map induced by the inclusion maps. The fundamental theorem of PH is as follows.
\begin{theoremext}\label{thm:ph-fundamental}
  There exists a unique decomposition of $H_k(\X; \Bbbk)$ ,
  \begin{equation}
    H_k(\X; \Bbbk) \simeq \bigoplus_{m=1}^L I(b_m, d_m), 
  \end{equation}
  where $0 < b_m < d_m \leq \infty$, $\Bbbk \to \Bbbk$ is an identity map, $0 \to \Bbbk, \Bbbk \to 0$ are zero maps, and $I(b, d)$ are as shown below:
  \begin{align}
    I(b, d) = & 0 \to \cdots \to 0 \to \overbrace{\Bbbk}^{b\text{th}} \to \cdots \to \overbrace{\Bbbk}^{(d-1)\text{th}} \to \overbrace{0}^{d\text{th}}
    \to \cdots \to 0, \ \  \text{ if } d_m \not = \infty, \\
    I(b, d) = & 0 \to \cdots \to 0 \to \overbrace{\Bbbk}^{b\text{th}} \to \cdots \to \Bbbk, \ \  \text{ if } d_m = \infty.
  \end{align}
  
\end{theoremext} 

On an intuitive level, the map $0 \to \Bbbk$ indicates the appearance of a homology generator, $\Bbbk \to \Bbbk$ indicates the persistence of the generator, and $\Bbbk \to 0$ indicates the disappearance of the generator.

We note that we should use a field as a coefficient ring of the homology module since the theorem holds only when $\Bbbk$ is a field.

\subsection{Order with level}\label{subsec:order-with-level}

We earlier introduced the idea of a level function. We now introduce the concept of total order consistent with the level function. Total order is used to rigorously describe the algorithms; the level function is needed to rigorously describe Theorem~\ref{thm:optimal-stable-codim1}. Simplices with the same level often appear in an alpha filtration or a Vietoris-Rips filtration; thus, the concept of total order is required to deal with such filtrations.
The pair consisting of the total order and the level function is called an \emph{order with level}. 

Let $\prec$ be a total order on the set of all simplices of $X$. We assume that
\begin{equation}\label{eq:simplices-ordering}
  \sigma \subsetneq \sigma' \text{ implies } \sigma \prec \sigma'.
\end{equation}
We can number all simplices in $X$ by this ordering as follows:
\begin{equation}\label{eq:X-order}
  \sigma_1 \prec \sigma_2 \prec \cdots \prec \sigma_N.
\end{equation}
From the condition \eqref{eq:simplices-ordering}, $X_i = \{\sigma_1, \ldots, \sigma_i\}$ is always
a subcomplex of $X$ and
\begin{equation}\label{eq:X-filtration}
  \emptyset = X_0 \subsetneq X_1 \subsetneq \ldots \subsetneq X_N = X
\end{equation}
is a filtration of a simplicial complex. In this filtration, the number of simplices is increased one by one, which simplifies the mathematical description. 

To include additional information in the filtration, we consider an order with level.
\begin{definition}
  A pair $r = (\hat{r}, \preceq_r)$ of a real-valued map on $X$ and a total order on $X$ satisfying the following conditions is called an \emph{order with level}.
  \begin{enumerate}
  \item $\sigma \preceq_r \sigma'$ implies $\hat{r}(\sigma) \leq \hat{r}(\sigma')$
  \item The order $\preceq_r$ satisfies \eqref{eq:simplices-ordering}
  \end{enumerate}
\end{definition}

From the total order $\preceq_r$, a filtration $\X_r$ is defined by \eqref{eq:X-order} and \eqref{eq:X-filtration}.
Theorem~\ref{thm:ph-fundamental} gives the unique decomposition
$H_k(\X_r; \Bbbk) \simeq \bigoplus_{m=1}^L I(b_m, d_m)$. For each $(b_m, d_m)$,
the simplex $\sigma_{b_m}$ is called a \emph{birth simplex}, $\sigma_{d_m}$ is called a
\emph{death simplex}, and the pair $(\sigma_{b_m}, \sigma_{d_m})$ is called a \emph{birth-death simplices pair}. When $d_m = \infty$, we write the pair as $(\sigma_{b_m}, \star)$ using the special symbol $\star$. $\bdsp_k(\X_r)$ denotes the set of all birth-death simplices pairs.
Moreover, $\hat{r}(\sigma_{b_m})$ is called a \emph{birth time}, $\hat{r}(\sigma_{d_m})$ is called
a \emph{death time}, and the pair of the birth time and death time, 
$(\hat{r}(\sigma_{b_m}), \hat{r}(\sigma_{d_m}))$, is called a \emph{birth-death pair}.
When $\sigma_{d_m} = \star$, $\hat{r}(\star)$ is defined as $+\infty$.
The multiset of birth-death pairs is called a \emph{persistence diagram}:
\begin{equation}
  \{(\hat{r}(\sigma_{b_m}), \hat{r}(\sigma_{d_m})) \mid m=1,\ldots,L; \hat{r}(\sigma_{b_m}) \not = \hat{r}(\sigma_{d_m})\}.
\end{equation}
We write the persistence diagram as $\pd_k(X, r)$.

We can readily show the following fact.
\begin{remark}
  For any $(\tau_0, \omega_0) \in \bdsp_k(\X_r)$, $\tau_0$ is a $k$-simplex and $\sigma_0$ is a $(k+1)$-simplex.
\end{remark}


By using the level information, we can describe the stability theorem of persistence diagrams~\cite{Cohen-Steiner2007,Chazal2009,Bauer2014,Lesnick2015}. Let $\pd_1, \pd_2$ be two persistence diagrams. The bottleneck distance between two diagrams $d_B(\pd_1, \pd_2)$ is defined as
\begin{equation}
  \label{eq:bottleneck}
  d_B(\pd_1, \pd_2) = \inf_{\gamma} \sup_{x \in \pd_1 \cup \Delta} \|x - \gamma(x)\|_{\infty},
\end{equation}
where $\gamma$ ranges over all bijections from $\pd_1 \cup \Delta$ to $\pd_2 \cup \Delta$, $\Delta$ is the diagonal, and $\|\cdot\|_\infty$ is the $\ell^\infty$ norm. In our setting, the stability theorem of PH is described as below.
\begin{theoremext}[Stability theorem of PH]\label{thm:stability-ph}
  Let $X$ be a finite simplicial complex. For any two orders with levels $q =(\hat{q}, \preceq_q)$ and $r = (\hat{r}, \preceq_r)$ and any $k \geq 0$, the following inequality holds.
  \begin{equation}
    d_B(\pd_k(X, q), \pd_k(X, r)) \leq \|\hat{q} - \hat{r}\|_\infty := \max_{\sigma \in X} |\hat{q}(\sigma) - \hat{r}(\sigma)|
  \end{equation}
\end{theoremext}
The theorem ensures the robustness of PDs to noise.

\subsection{Optimal volumes}

Optimal volume is proposed in~\cite{voc} as a way to extract homological structures corresponding to a birth-death pair. For an order with levels $r = (\hat{r}, \preceq_r)$ and a birth-death simplices pair $(\tau_0, \omega_0) \in \bdsp_k(\X_r)$ with $\omega_0 \not = \star$, the \emph{optimal volume} for the pair is formalized as the solution to the following minimization problem:
\begin{equation}
  \begin{aligned}
    &\text{minimize } \|z\|_0, \text{ subject to } \\
    z &= \omega_0 + \sum_{\omega \in \mathcal{F}_{k+1}} \alpha_\omega \omega \in C_{k+1}(X; \Bbbk), \\
    \tau^*(\partial z) &= 0 \text{ for every } \tau \in \mathcal{F}_k, \\
    \tau_0^*(\partial z) & \not = 0,
  \end{aligned}
\end{equation}
where $C_{k+1}(X; \Bbbk)$ is a chain complex whose coefficient field is $\Bbbk$, $\mathcal{F}_k = \{ \sigma \in X^{(k)} \mid \tau_0 \prec_r \sigma \prec_r \omega_0\}$, $\tau^*$ is an element of the dual basis of cochain complex $C^k(X; \Bbbk)$, and $\|z\|_0 = \{ \omega \mid \alpha_\omega \not = 0 \}$ is the $\ell^0$ norm of $z$. For solution $z$, $\partial z$ is called a \emph{volume-optimal} cycle.

Reference~\cite{voc} shows the following fact, which indicates that the volume-optimal cycle is suitable for representing a birth-death pair.
\begin{fact}
  Let $(\tau_0, \omega_0) \in \bdsp_k(\X_r)$ be a birth-death simplices pair and $z$ an optimal volume of that pair. Then the following relations hold:
  \begin{align}
    \partial z & \not \in Z_k(\{\sigma \in X \mid \sigma \prec_r \tau_0\}), \\
    \partial z & \in Z_k(\{\sigma \in X \mid \sigma \preceq_r \tau_0\}), \\
    \partial z & \not \in B_k(\{\sigma \in X \mid \sigma \prec_r \sigma_0\}), \\
    \partial z & \in B_k(\{\sigma \in X \mid \sigma \preceq_r \sigma_0\}),
  \end{align}
  where $Z_k(\cdot)$ represents the cycles and $B_k(\cdot)$ represents the boundaries.
\end{fact}

\subsection{Optimal volumes for the $(n-1)$th PH and persistence trees}

In this section, we consider a triangulation of a convex set in $\R^n$ and its $(n-1)$th PH with $n \geq 2$. More precisely, we assume the following:
\begin{cond}\label{cond:rn}
  A simplicial complex $X$ in $\R^n$ satisfies two conditions.
  \begin{itemize}
  \item Any $k$-simplex ($k < n$) is a face of an $n$-simplex
  \item $|X| := \bigcup_{\sigma \in X}\sigma  $ is convex in $\R^n$
  \end{itemize}
\end{cond}

Schweinhart \cite{Ben-thesis} pointed out that the $(n-1)$th PH was isomorphic to the 0th persistent cohomology of the dual filtration
by Alexander duality. Zeroth cohomology is deeply related to the connected components in the dual filtration. This gives rise to another formalization of $(n-1)$th PH. Reference \cite{voc} generalizes the idea shown in \cite{Ben-thesis}.

We now examine an order with levels $r$ on $X$. Consider the filtration $\X_r$ in \eqref{eq:X-order} and \eqref{eq:X-filtration} given by $\preceq_r$. For simplicity, we will use $\Z_2$ as the coefficient field. Then the following proposition and theorems hold~\cite{Ben-thesis, voc}.

\begin{propext}
  For any birth-death pair in $\pd_{n-1}(X, r)$, the death time is not infinity.
\end{propext}

\begin{theoremext}
  The optimal volume for any birth-death pair is uniquely determined.
\end{theoremext}
  
\begin{theoremext}\label{thm:vochd-tree}
  If $z$ and $z'$ are the optimal volumes for two different birth-death pairs,
  one of the following holds:
  \begin{itemize}
  \item $z \cap z' = \emptyset$
  \item $z \subsetneq z'$
  \item $z \supsetneq z'$
  \end{itemize}
  Note that we can naturally regard any
  $z = \sum_{\omega \in X^{(n)}} k_{\omega} \sigma \in C_{n}(X)$ as a subset of $n$-simplices of $X$,
  $\{\omega \in X^{(n)} \mid k_{\omega} \not = 0\}$,
  since we use $\Z_2$ as the homology coefficient field.
\end{theoremext}

From Theorem~\ref{thm:vochd-tree}, we know that
$D_{n-1}(\X_r)$ can be regarded as a forest (i.e., the union of distinct trees)
by the inclusion relation. In~\cite{Ben-thesis}, the forest is called \textit{persistence trees}.
Moreover, we can compute the persistence trees by the merge-tree algorithm (Algorithm~\ref{alg:volopt-hd-compute}).

To describe the algorithm, consider the one-point compactification space $\R^n \cup \{\infty\} \simeq S^n$.
The following facts are well known.
\begin{fact}
  $X \cup \{\omega_\infty\}$ is a cell decomposition of $\R^n \cup \{\infty\}$ where
  $\omega_\infty = \R^n \cup \{\infty\} \backslash |X|$.
\end{fact}

\begin{fact}
  For any $\tau \in X^{(n-1)}$, the proper cofaces\footnote{ If $\tau$ is a face of $\omega$, $\omega$ is called a \emph{coface} of
  $\tau$. If the difference of the dimensions is one, $\omega$ is a \emph{proper coface} of $\tau$.} of $\tau$ are just two $n$-cells in $X \cup \{\omega_\infty\}$.
\end{fact}

We can extend the order with levels $r$ onto $X \cup \{\omega_\infty\}$ by regarding $\omega_\infty$ as the maximum element and
$\hat{r}(\omega_\infty) = +\infty$.
To describe the algorithm, 
we consider a directed graph $\pt$ whose nodes are $n$-cells in $X \cup \{\omega_{\infty}\}$.
An edge has extra data
in $X^{(n-1)}$, and we can write the edge from $\omega$ to $\omega'$ with
extra data $\tau$ as $(\omega \xrightarrow{\tau} \omega')$.
The directed graph $\pt$ is increasingly updated in Algorithm~\ref{alg:volopt-hd-compute}.

Since the graph is always a forest throughout the algorithm\cite{voc},
we can find a root
of a tree that contains an $n$-cell $\omega$ in the graph $\pt$
by recursively following the edges from $\omega$.
We call this procedure \textproc{Root}($\omega, \pt$).

\begin{algorithm}
  \caption{Computing persistence trees by the merge-tree algorithm.}\label{alg:volopt-hd-compute}
  \begin{algorithmic}
    \Procedure{Compute-Tree}{$\X_r$}
    \State initialize $\pt = \{\omega_\infty\}$
    \For{$\sigma \in X$ in $(\preceq_r)$-descending order} \Comment (LOOP)
      \If{$\sigma$ is an $n$-simplex}
        \State add $\sigma$ to $\pt$ as a vertex
      \ElsIf{$\sigma$ is an $(n-1)$-simplex}
        \State let $\omega_1$ and $\omega_2$ be two cofaces of $\sigma$
        \State $\omega_1' \gets \textproc{Root}(\omega_1, \pt)$
        \State $\omega_2' \gets \textproc{Root}(\omega_2, \pt)$
        \If{$\omega_1' = \omega_2'$}
          \State \textbf{continue}
        \ElsIf{$\omega_1' \succ_r  \omega_2'$}
          \State Add $(\omega_2' \xrightarrow{\sigma} \omega_1')$ to $\pt$ as an edge
        \Else
          \State Add $(\omega_1' \xrightarrow{\sigma} \omega_2')$ to $\pt$ as an edge
        \EndIf
      \EndIf
    \EndFor
    \Return $\pt$
    \EndProcedure
  \end{algorithmic}
\end{algorithm}

The following theorem gives the interpretation of
the result of the algorithm as persistence information.
\begin{theoremext}\label{thm:vochd-alg}
  Let $\pt_*$ be the result of Algorithm~\ref{alg:volopt-hd-compute}. Then
  the following holds:
  
  \begin{enumerate} 
  \item $\pt_*$ is a tree whose root is $\omega_\infty$
  \item $D_{n-1}(\X_r) = \{(\tau, \omega) \mid (\omega \xrightarrow{\tau} *) \text{ is an edge of } \pt_*\}$. Here
    $*$ means another vertex
  \item If there is an edge $\omega' \xrightarrow{\tau} \omega$ is in $\pt_*$,  we have $\omega' \prec_r \omega$ 
  \item The optimal volume for $(\tau, \omega) \in D_{n-1}(\X_r)$ is $\mathrm{dec}(\omega, \pt_*)$, where
    $\mathrm{dec}(\omega, \pt_*)$ the set of all descendant nodes of $\omega$ in $\pt_*$ including $\omega$ itself
  \item $\pt_*$ gives the persistence trees. That is,
    $(\tau, \omega)$ is the parent of $(\tau', \omega')$ in the persistence trees
    if and only if there are edges
    $\omega' \xrightarrow{\tau'} \omega$
  \end{enumerate}
\end{theoremext}

We can prove the above theorems using Alexander duality. Appendix A in~\cite{voc*}\footnote{ This paper is the preprint version of
\cite{voc}. The contents of Appendix A in \cite{voc*} are omitted from~\cite{voc}; thus, we sometimes refer to the preprint version.}
provides a detailed discussion.


\section{Stable volumes for the (n-1)th PH}\label{sec:stable-volume-n-1}

We can now describe stable volume for $(n-1)$th PH using persistence trees under Condition~\ref{cond:rn}.
The parameter $\epsilon$ in the following definition is called a \emph{bandwidth parameter}.

\begin{definition}\label{def:sv-tree}
  Let $X$ be a simplicial complex and $r$ an order with levels on $X$, and $\X_r$ be
  the filtration of $X$ given by \eqref{eq:X-order} and \eqref{eq:X-filtration}.
  Let $\pt_*$ be persistence trees of the filtration $\X_r$.
  For $(\tau_0, \omega_0) \in D_{n-1}(\X_r)$ and a positive number $\epsilon$, the stable volume of the birth-death simplices pair
  with a noise bandwidth parameter $\epsilon$,
  $\SV_\epsilon(r, \tau_0, \omega_0)$,
  is defined as follows:
  \begin{equation}\label{eq:stablevolume}
    \SV_\epsilon(r, \tau_0, \omega_0) = \{\omega_0\} \cup
    \left(\bigcup_{\omega \in C_\epsilon(\tau_0, \omega_0)} \mathrm{dec}(\omega, \pt_*) \right),
  \end{equation}
  where $C_\epsilon(\tau_0, \omega_0)$ is given as 
  \begin{align}
    C_\epsilon(\tau_0, \omega_0) &= \{ \omega \in V \mid 
    \omega \xrightarrow{\tau} \omega_0,\text{ and }
    \hat{r}(\tau) \geq \hat{r}(\tau_0) + \epsilon
    \}.  \label{eq:c-epsilon}
  \end{align}
\end{definition}
We specify the following \emph{$(r, \omega_0)$-order condition} in order to describe
the main theorem.
\begin{definition}
  Let $r$ and $q$ be two orders with levels on a simplicial complex $X$
  and $\omega_0$ be an $n$-simplex.
  Then $q$ satisfies an \emph{$(r, \omega_0)$-order condition} if
  $ \sigma \prec_{r} \omega_0$ implies $\sigma \prec_q \omega_0$
  for all $\sigma \in X$.
\end{definition}

We also define the symbol $\tau_{q, \omega_0}$.
For an order with levels $q$ and an $n$-simplex $\omega_0$,
$\tau_{q, \omega_0}$ is the birth simplex paired with $\omega_0$ in $D_{n-1}(\X_q)$.
That is, $(\tau_{q, \omega_0}, \omega_0) \in D_{n-1}(\X_q)$.

The following theorem is the first main theorem of the paper.
It states that the stable volume is an invariant part of optimal volumes in the presence of small noise.

\begin{theorem}\label{thm:stable-tree}
  \begin{equation}
    \SV_{\epsilon}(r, \tau_0, \omega_0) = \bigcap_{q \in \mathcal{R}_\epsilon} \OV(q, \omega_0),
  \end{equation}
  where
  \begin{equation}
    \begin{aligned}
      \mathcal{R}_\epsilon = \{& q = (\hat{q}, \preceq_{q}): \text{ an order with levels } \mid \\
      & \|\hat{q} - \hat{r}\|_{\infty} < \epsilon / 2 \text{ and }
      q \text{ satisfies $(r, \omega_0)$-order condition} \},
    \end{aligned} 
  \end{equation}
  and
  \begin{equation}
    \begin{aligned}
      \OV(q, \omega_0) & =  \text{the optimal volume of } (\tau_{q, \omega_0}, \omega_0) \in D_{n-1}(\X_q).
    \end{aligned}
  \end{equation}
\end{theorem}
The theorem treats two different filtrations, $\X_r$ and $\X_q$, on the same simplicial complex $X$. It should be noted that the reader needs to treat these filtrations carefully.

The following two claims immediately imply Theorem~\ref{thm:stable-tree} and are discussed in the next three subsections.
\begin{claim}\label{claim:easy}
  For any $q \in \mathcal{R}_\epsilon$, the following relation holds:
  \begin{equation}\label{eq:sv_subset_ov}
    \SV_\epsilon(r, \tau_0, \omega_0) \subseteq \OV(q, \omega_0).
  \end{equation}
\end{claim}

\begin{claim}\label{claim:hard}
  For any $\tilde{\omega} \not\in \SV_\epsilon(r, \tau_0, \omega_0)$,
  there is an order with levels $q \in \mathcal{R}_\epsilon$ satisfying
  $\tilde{\omega}  \not \in \OV(q, \omega_0)$.
\end{claim}

\subsection{Dual graph and its subgraphs}

To show the theorem, we introduce a dual graph of $X$ and its subgraphs.
\begin{definition}
  $V_X$ and $E_X$ are defined as follows:
  \begin{equation}
    \begin{aligned}
      V_X &= X^{(n)} \cup \{\omega_\infty\}, \\
      E_X &= X^{(n-1)}.
    \end{aligned}
  \end{equation}
  $G_X = V_X \cup E_X$ is an undirected graph when the two endpoints of $\tau \in E_X$ are defined as
  the two cofaces of $\tau$.

\end{definition}

We call the graph $G_X$ the \emph{dual graph} of $X$.
We define the subgraphs of $G_X$, $\isubg{r}{\sigma}$ and
$G_X(\hat{r} \geq s)$, as 
\begin{equation}
  \begin{aligned}
    \isubg{r}{\sigma} &= \isubgv{r}{\sigma} \cup \isubge{r}{\sigma}  \\
    \isubgv{r}{\sigma} &= \{\omega \in V_X \mid
   \omega \succeq_r \sigma \},\\
    \isubge{r}{\sigma} &= \{ \tau \in E_X \mid
    \tau \succeq_r \sigma \}, \\
    G_X(\hat{r} \geq s) &= V_X(\hat{r} \geq s) \cup E_X(\hat{r} \geq s), \\
    V_X(\hat{r} \geq s) &= \{\omega \in V_X \mid \hat{r}(\omega) \geq s \}, \\
    E_X(\hat{r} \geq s) &= \{ \tau \in E_X \mid \hat{r}(\tau) \geq s \},
  \end{aligned}
\end{equation}
where $\sigma$ is a simplex in $X$ and $s$ is a real number.
The condition \eqref{eq:simplices-ordering} ensures that $\isubg{r}{\sigma}$ and
$G_X(\hat{r} \geq s)$
are subgraphs of $G_X$. We also define $\subg{r}{\sigma}$ by replacing $\succeq_r$ with $\succ_r$.

We introduce the notation $\pt_\sigma$ as $\pt$ in Algorithm~\ref{alg:volopt-hd-compute} when the inside of the (LOOP) is finished at $\sigma$. 
The following propositions are essential
for the proof of Theorem~\ref{thm:stable-tree}.
The facts are shown as Fact 17 and Fact 18 in \cite{voc*}.
\begin{fact}\label{fact:connectivity}
  The topological connectivity of vertices 
  in $\pt_\sigma$ is the same as $\isubg{r}{\sigma}$. That is,
  $\omega_1, \ldots, \omega_k \in \isubgv{r}{\sigma}$ are all vertices
  of a connected component in $\isubg{r}{\sigma}$ if and only if
  there is a tree in $\pt_\sigma$ whose vertices are
  $\omega_1, \ldots, \omega_k$.
\end{fact}

\begin{fact}\label{fact:root}
  For each tree in $\pt_\sigma$, the root of the tree is
  the $(\preceq_r)$-maximum vertex in the tree.
\end{fact}

The next proposition describes the role of the birth and death simplices. 
$K(G, \omega)$ denotes the connected component of a graph $G$ which
contains a vertex $\omega$. $\Kv(G, \omega)$ and $\Ke(G, \omega)$ denote
all the vertices and all the edges of $K(G, \omega)$.

\begin{prop}\label{prop:connect-two-comp}
  Let $(\tau_0, \omega_0) \in \bdsp_{n-1}(\X_r)$.
  Then $\omega_0$ is the $(\preceq_r)$-maximum $n$-simplex in
  $K(\subg{r}{\tau_0}, \omega_0)$ and
  the optimal volume of $(\tau_0, \omega_0)$ is $K(\subg{r}{\tau_0}, \omega_0)$.
  Furthermore, 
  there exists $\omega_1 \in V_X$ satisfying
  the following conditions:
  \begin{itemize}
  \item $\omega_0 \xrightarrow{\tau_0} \omega_1$ in $\pt_*$
  \item $\omega_1 \succ_r \omega_0$
  \item $K(\isubg{r}{\tau_0}, \omega_0) = K(\subg{r}{\tau_0}, \omega_0) \cup
    K(\subg{r}{\tau_0}, \omega_1) \cup \{\tau_0\}$.
    That is, $\tau_0$ is the edge between
    $K(\subg{r}{\tau_0}, \omega_0)$ and $K(\subg{r}{\tau_0}, \omega_1)$
  \end{itemize} 
  
\end{prop}

We can easily prove the proposition by considering Algorithm~\ref{alg:volopt-hd-compute}
and Facts~\ref{fact:connectivity} and \ref{fact:root}. 

\subsection{Proof of Claim \ref{claim:easy}}

The following equality holds
from the definition of stable volume \eqref{eq:stablevolume}, Theorem~\ref{thm:vochd-alg},
Facts \ref{fact:connectivity} and \ref{fact:root}, and Proposition \ref{prop:connect-two-comp}.
\begin{equation}\label{eq:sv-by-kv}
  \SV_\epsilon(r, \tau_0, \omega_0) = \Kv(G_X(\hat{r} \geq \hat{r}(\tau_0) + \epsilon), \omega_0).
\end{equation}
The following equality also holds from Proposition \ref{prop:connect-two-comp}.
\begin{equation}
  \OV(q, \omega_0) = \Kv(\subg{q}{\tau_{q,\omega_0}}, \omega_0).
\end{equation}
Therefore, we can show the following relationship to prove Claim~\ref{claim:easy}.
\begin{equation}\label{eq:Kv-inclusion-relation}
  \Kv(G_X(\hat{r} \geq \hat{r}(\tau_0) + \epsilon), \omega_0) \subseteq
  \Kv(\subg{q}{\tau_{q,\omega_0}}, \omega_0).
\end{equation}

The following lemma is essential to prove Claim~\ref{claim:easy}.
\begin{lemma}\label{lem:birth-stability}
  Let $(\tau_0, \omega_0) \in D_{n-1}(\X_r)$ be a birth-death simplices pair and
  $q \in \mathcal{R}_\epsilon$.
  Then the following inequality holds:
  \begin{equation}
    \hat{q}(\tau_{q, \omega_0}) - \epsilon / 2 < \hat{r}(\tau_0).
  \end{equation}
\end{lemma}

\begin{proof}
  We assume that $\hat{q}(\tau_{q, \omega_0}) - \epsilon / 2 \geq \hat{r}(\tau_0)$
  and will show a contradiction.
  Since $\|\hat{q} - \hat{r}\|_{\infty} < \epsilon / 2$,
  for any $\sigma \in G_X$
  with $\tau_{q, \omega_0} \preceq_q \sigma$,
  we have
  \begin{equation}
    \hat{r}(\tau_0) \leq \hat{q}(\tau_{q, \omega_0}) - \epsilon / 2
    \leq \hat{q}(\sigma) - \epsilon /2
    < (\hat{r}(\sigma) + \epsilon/2) - \epsilon /2
    = \hat{r}(\sigma).
  \end{equation}
  From the definition of order with levels,
  the inequality immediately implies
  $\tau_0 \prec_r \sigma$ and
  $\isubg{q}{\tau_{q, \omega_0}}$ is a subgraph of $\subg{r}{\tau_0}$.

  Since $(\tau_{q, \omega_0}, \omega_0) \in \bdsp_{n-1}(\X_q)$,
  Proposition~\ref{prop:connect-two-comp} gives $\omega_1 \in V_X$
  satisfying
  \begin{align}
    \omega_1 & \succ_q \omega_0, \label{eq:q-order} \\
    K(\isubg{q}{\tau_{q, \omega_0}}, \omega_0) &=
    K(\subg{q}{\tau_{q, \omega_0}}, \omega_0) \cup
    K(\subg{q}{\tau_{q, \omega_0}}, \omega_1) \cup
    \{\tau_{q, \omega_0}\}. \label{eq:q-conn-two}
  \end{align}
  \eqref{eq:q-order} leads to $\omega_1 \succ_r \omega_0$
  from the $(r, \omega_0)$-order condition. At the same time,
  \eqref{eq:q-conn-two} leads to $\omega_1 \in K(\subg{r}{\tau_0}, \omega_0)$
  since $\isubg{q}{\tau_{q, \omega_0}}$ is a
  subgraph of $\subg{r}{\tau_0}$.
  These facts lead to a contradiction since $\omega_0$ is 
  the $(\preceq_r)$-maximum vertex in $K(\subg{r}{\tau_0}, \omega_0)$,
  but $K(\subg{r}{\tau_0}, \omega_0)$ contains $\omega_1$ and
  $\omega_1 \succ_r \omega_0 $.
\end{proof}

Using Lemma~\ref{lem:birth-stability}, we prove that
$G_X(\hat{r} \geq \hat{r}(\tau_0) + \epsilon)$ is a subgraph of
$\subg{q}{\tau_{q, \omega_0}}$ to show \eqref{eq:Kv-inclusion-relation}.
For any $ \sigma \in G_X(\hat{r} \geq \hat{r}(\tau_0) + \epsilon)$, we have
\begin{equation}
  \hat{q}(\sigma) > \hat{r}(\sigma) - \epsilon / 2 \geq
  (\hat{r}(\tau_0) + \epsilon) - \epsilon / 2
  = \hat{r}(\tau_0) + \epsilon/2
  > \hat{q}(\tau_{q, \omega_0}),
\end{equation}
and hence $\sigma \succ_q \tau_{q, \omega_0}$. This means that
$G_X(\hat{r} \geq \hat{r}(\tau_0) + \epsilon)$ is a subgraph of
$\subg{q}{\tau_{q, \omega_0}}$.

\subsection{Proof of Claim \ref{claim:hard}}

If $\tilde{\omega} \not \in OV(r, \omega_0)$, the conclusion of Claim~\ref{claim:hard}
is trivial. Therefore, we assume $\tilde{\omega} \in OV(r, \omega_0)$.

Since
\begin{align*}
  \tilde{\omega} &\in \OV(r, \omega_0) = \Kv(\subg{r}{\tau_0}, \omega_0), \\
  \tilde{\omega} &\not \in \SV_\epsilon(r, \tau_0, \omega_0) = \Kv(G_X(\hat{r} \geq \hat{r}(\tau_0) + \epsilon), \omega_0), 
\end{align*}
and $G_X(\hat{r} > \hat{r}(\tau_0) + \epsilon)$ is a subgraph of $\subg{r}{\tau}$,
there is $\tilde{\tau} \in E_X$ satisfying the following conditions:
\begin{equation} \label{eq:tau-start}
  \begin{aligned}
    \tau_0 & \prec_r \tilde{\tau}, \\
    \hat{r}(\tau_0) & \leq \hat{r}(\tilde{\tau}) < \hat{r}(\tau_0) + \epsilon, \\
    \tilde{\omega} & \not \in \Kv(\subg{r}{\tilde{\tau}}, \omega_0),  \\
    \tilde{\omega} & \in \Kv(\isubg{r}{\tilde{\tau}}, \omega_0).
  \end{aligned}
\end{equation}
Let $\epsilon_* = \hat{r}(\tilde{\tau}) - \hat{r}(\tau_0)$ and $\eta = (\epsilon_* + \epsilon) /4$.
From \eqref{eq:tau-start}, we have
\begin{equation}
  \begin{aligned}
    \epsilon /2 & > \eta > \epsilon_* / 2 \geq 0.
  \end{aligned}
\end{equation}

Proposition~\ref{prop:connect-two-comp} gives $\omega_1 \in V_X$ satisfying
\begin{align}
  \omega_1 &\succ_r \omega_0, \text{ and } \label{eq:omega-1-succ-omega-0} \\
  K(\isubg{r}{\tau}, \omega_0) & = K(\subg{r}{\tau_0}, \omega_0) \cup
                                   K(\subg{r}{\tau_0}, \omega_1) \cup
                                   \{\tau_0\}.
\end{align}
Let  $\omega_e$ be the endpoint of $\tau_0$ contained in
$K(\subg{r}{\tau_0}, \omega_0)$. A path from $\omega_0$ to $\omega_e$
exists in $K(\subg{r}{\tau_0}, \omega_0)$. We consider the following
two cases regarding the path.
\begin{description}
\item[Case 1] There is a path $P$ from $\omega_0$ to $\omega_e$ in $\subg{r}{\tau_0}$
  without passing through $K(\subg{r}{\tilde{\tau}}, \tilde{\omega})$.
\item[Case 2] Any path
  $\omega_0$ to $\omega_e$ in $K(\subg{r}{\tau_0}$ passes through
  $K(\subg{r}{\tilde{\tau}}, \tilde{\omega})$.
\end{description}
We will divide the proof into the above two cases.

\subsubsection{Case 1}
Figure~\ref{fig:case-1-figure} shows the relationship between
connected components and the path $P$ in Case 1. We can construct an order with levels $q$ satisfying
$\tilde{\omega} \not \in \Kv(\subg{q}{\tau_{q, \omega_0}}, \omega_0)$ based on this relationship.

\begin{figure}[htbp]
  \centering
  \includegraphics[width=0.8\hsize]{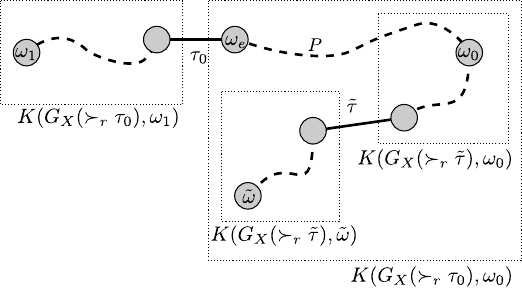}
  \caption{
    The relationship between connected components and path $P$ in
    Case 1.
    Circles represent vertices, rectangles by dotted lines represent
    connected components, solid lines represent edges between connected
    components, and dashed curves represent paths.
  }
  \label{fig:case-1-figure}
\end{figure}

Using $\eta$, we define a function $\hat{q}$ on $X$ as follows:
\begin{equation}\label{eq:def-hatq-case-1}
  \hat{q}(\sigma) = \left\{
    \begin{array}{ll}
      \hat{r}(\sigma) + \eta & \text{if } \sigma \in X^{(n)},  \\
      \hat{r}(\sigma) + \eta & \text{if } \sigma \in Q_{\mathrm{E}}, \\
      \hat{r}(\sigma) - \eta & \text{if } \sigma \in X^{(n-1)} \backslash Q_{\mathrm{E}}, \\
      \hat{r}(\sigma) - \eta & \text{if } \sigma \in X^{(0)} \cup \cdots \cup X^{(n-2)},
    \end{array} \right.
\end{equation}
where
$Q = P \cup \{\tau_0\} \cup K(\subg{r}{\tau}, \omega_1)$, and
$Q_{\mathrm{E}}$ is the set of all edges in $Q$.
We can easily show that $\hat{q}$ is a level function from the fact that
$\hat{r}$ is a level function.
We define the order on $X \cup \{\omega_\infty\}$, $\preceq_q$, as follows:
\begin{equation}
  \label{eq:def-q-case-1}
  \begin{aligned}
    &\sigma \prec_q \sigma' \text{ if and only if} \\
    &\hat{q}(\sigma) < \hat{q}(\sigma'), \text{ or }\\
    &\hat{q}(\sigma) = \hat{q}(\sigma') \text{ and } \sigma \subsetneq \sigma', \text{ or } \\
    &\hat{q}(\sigma) = \hat{q}(\sigma') \text{ and } \sigma \cap \sigma' = \emptyset \text{ and }
    \sigma \prec_r \sigma'.
  \end{aligned}
\end{equation}
This is a kind of lexicographic order by the total preorder $(\sigma, \sigma') \mapsto \hat{q}(\sigma) \leq \hat{q}(\sigma')$,
the partial order $\subseteq$, and the total order $\preceq_r$. Therefore $\preceq_q$ is
a total order.

The following two facts are easily shown.

\begin{fact}\label{fact:case-1-q-owl}
  $q = (\preceq_q, \hat{q})$ is an order with levels.
\end{fact}

\begin{fact}\label{fact:order-coincidence-q}
  The order $\preceq_q$ coincides with $\preceq_r$ on $V_X$. The same is true on $Q_{\mathrm{E}}$ or $X^{(n-1)} \backslash Q_{\mathrm{E}}$.
\end{fact}

\begin{fact}\label{fact:q-in-R-epsilon}
  $q \in \mathcal{R}_\epsilon$.
\end{fact}

\begin{proof}
  From the definition, we have $\|\hat{q} - \hat{r}\| \leq \eta < \epsilon / 2$. To show the $(r, \omega_0)$-order condition,
  we assume $\sigma \prec_r \omega_0$ and show $\sigma \prec_r \omega_0$. From the assumption,
  we have $\hat{r}(\sigma) \leq \hat{r}(\omega_0)$.
  Since $\omega_0$ is $n$-simplex, $\hat{q}(\omega_0) = \hat{r}(\omega_0) + \eta$, and so $\hat{q}(\omega_0) \geq \hat{q}(\sigma)$.
  Therefore, we consider the following cases:
  \begin{enumerate}
  \item $\hat{q}(\omega_0) > \hat{q}(\sigma)$. In this case, we can see $\omega_0 \succ_q \sigma $ from the definition of
    order with levels.
  \item $\hat{q}(\omega_0) = \hat{q}(\sigma)$ and $\omega_0 \supsetneq \sigma$. In this case, we have $\omega_0 \succ_q \sigma $
    from the definition of $\preceq_q$.
  \item $\hat{q}(\omega_0) = \hat{q}(\sigma)$ and $\omega_0 \subsetneq \sigma$. In this case, the condition $\omega_0 \subsetneq \sigma$ 
    leads to a contradiction since $\prec_r$ satisfies \eqref{eq:simplices-ordering} but $\sigma \prec_r \omega_0$.
  \item $\hat{q}(\omega_0) = \hat{q}(\sigma)$ and $\omega_0 \cap \sigma = \emptyset$. In this case, we have $\omega_0 \succ_q \sigma$
    from the assumption of $\sigma \prec_r \omega_0$.
  \end{enumerate}
  In all cases except 3, we have $\sigma \prec_r \omega_0$.
\end{proof}

The following fact comes from Fact~\ref{fact:order-coincidence-q}.
\begin{fact}\label{fact:tau-0-q-minimum-in-Q}
  $\tau_0$ is the $(\preceq_q)$-minimum element in $Q$.
\end{fact}

The above fact leads to the following.
\begin{fact}\label{fact:P-subgraph}
  $Q$ is a subgraph of $K(\isubg{q}{\tau_0}, \omega_0)$.
\end{fact}

\begin{fact}\label{fact:case-1-omega-0-omega-1-differnt-cc}
  $K(\subg{q}{\tau_0}, \omega_0)$ and $K(\subg{q}{\tau_0}, \omega_1)$ are different connected components
  in $\subg{q}{\tau_0}$.
\end{fact}

\begin{proof}
  We assume that $K(\subg{q}{\tau_0}, \omega_0) = K(\subg{q}{\tau_0}, \omega_1)$ and this will lead to a contradiction.
  From the assumption, there is a path $Q'$ from $\omega_0$ to $\omega_1$ in $\subg{q}{\tau_0}$.
  Any edge $\tau$ on $Q'$ satisfies $\tau \succ_q \tau_0$ and so $\hat{q}(\tau) \geq \hat{q}(\tau_0)$. Therefore,
  since $|\hat{r}(\tau)  - \hat{q}(\tau)| \leq \eta$ and $\hat{q}(\tau_0) = \hat{r}(\tau_0) + \eta$,
  the inequality $\hat{r}(\tau) \geq \hat{r}(\tau_0)$ holds for any $\tau \in Q_{\mathrm{E}}'$,
  where $Q_{\mathrm{E}}'$ is the set of all edges on $Q'$.

  We conclude that $\tau \succ_r \tau_0$ for any $\tau \in Q_{\mathrm{E}}'$
  from \eqref{eq:def-q-case-1}, $\tau \succ_q \tau_0$, and $\hat{r}(\tau) \geq \hat{r}(\tau_0)$.
  This means that $Q'$ is a subgraph
  of $\subg{r}{\tau_0}$. However, this contradicts the fact that $\omega_0$ and $\omega_1$ are contained in
  different connected components of $\subg{r}{\tau_0}$.
\end{proof}

From Facts~\ref{fact:order-coincidence-q}, \ref{fact:P-subgraph}, and \ref{fact:case-1-omega-0-omega-1-differnt-cc}, and \eqref{eq:omega-1-succ-omega-0},
we have the following fact.
\begin{fact} \label{fact:case-1-tau-q-omega-0}
  $\tau_{q,\omega_0} = \tau_0$
\end{fact}

Next, we examine $\tilde{\tau}$.
\begin{fact}\label{fact:q-tau0-ge-tau}
  $\tau_0 \succ_q \tilde{\tau}$.
\end{fact}

\begin{proof}
  Since $\tau_0 \in Q_{\mathrm{E}}$ and $\tilde{\tau} \not \in Q_{\mathrm{E}}$, we have
  \begin{align*}
    \hat{q}(\tau_0) - \hat{q}(\tilde{\tau}) &= (\hat{r}(\tau_0) + \eta) - (\hat{r}(\tilde{\tau}) - \eta) \\
                                            &= \hat{r}(\tau_0) - \hat{r}(\tilde{\tau}) - 2\eta \\
                                            &= \epsilon_* - 2\eta > 0.
  \end{align*}
  The inequality leads to $\tau_0 \succ_q \tilde{\tau}$ since $q$ is an order with levels.
\end{proof}




We can prove the following fact in a similar way to Fact~\ref{fact:case-1-omega-0-omega-1-differnt-cc}.
\begin{fact}\label{fact:tilde-omega-omega-0-differnt-cc}
  $\tilde{\omega} \not \in K(\subg{q}{\tilde{\tau}}, \omega_0)$.
\end{fact}

From Facts~\ref{fact:case-1-tau-q-omega-0}, \ref{fact:q-tau0-ge-tau}, and \ref{fact:tilde-omega-omega-0-differnt-cc}, the following fact is easily established.
\begin{fact}
  $\tilde{\omega} \not \in K(\subg{q}{\tau_{q, \omega_0}}, \omega_0)$.
\end{fact}

Finally, since $\OV(q, \omega_0) = \Kv(\subg{q}{\tau_{q, \omega_0}}, \omega_0)$, we conclude that $\tilde{\omega} \not \in \OV(q, \omega_0)$.

\subsubsection{Case 2}

In Case 2, there is a path $T$ from $K(\subg{r}{\tilde{\tau}}, \tilde{\omega})$ to $\omega_e$ without passing through $K(\subg{r}{\tilde{\tau}}, \omega_0)$.
Figure~\ref{fig:case-2-figure} shows the relationship between connected components and paths between the components in this case.

\begin{figure}[htbp]
  \centering
  \includegraphics[width=0.8\hsize]{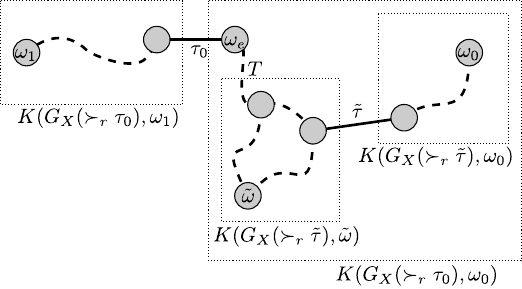}
  \caption{
    The relationship between connected components and path $T$ in Case 2.
    Circles represent vertices, rectangles by dotted lines represent
    connected components, solid lines represent edges between connected
    components, and dashed curves represent paths.
  }
  \label{fig:case-2-figure}
\end{figure}

We now define a function $\hat{s}$ on $X$ as follows:
\begin{equation}\label{eq:def-hats}
  \hat{s}(\sigma) = \left\{
    \begin{array}{ll}
      \hat{r}(\sigma) + \eta & \text{if } \sigma \in X^{(n)},  \\
      \hat{r}(\sigma) + \eta & \text{if } \sigma \in S_{\mathrm{E}}, \\
      \hat{r}(\sigma) - \eta & \text{if } \sigma \in X^{(n-1)} \backslash S_{\mathrm{E}}, \\
      \hat{r}(\sigma) - \eta & \text{if } \sigma \in X^{(0)} \cup \cdots \cup X^{(n-2)},
    \end{array} \right.
\end{equation}
where $S = K(\subg{r}{\tilde{\tau}}, \tilde{\omega}) \cup T \cup \{\tau_0\} \cup K(\subg{r}{\tau_0}, \omega_1)$ and
$S_{\mathrm{E}}$ is the set of all edges of $S$.

We also define the total order on $X$, $\preceq_s$, as follows:
\begin{equation}
  \label{eq:def-s}
  \begin{aligned}
    &\sigma \prec_s \sigma' \text{ if and only if} \\
    &\hat{s}(\sigma) < \hat{s}(\sigma'), \text{ or }\\
    &\hat{s}(\sigma) = \hat{s}(\sigma') \text{ and } \sigma \subsetneq \sigma', \text{ or } \\
    &\hat{s}(\sigma) = \hat{s}(\sigma') \text{ and } \sigma \cap \sigma' = \emptyset \text{ and }
    \sigma \prec_r \sigma'.
  \end{aligned}
\end{equation}

We can prove the following fact in a similar way to Facts~\ref{fact:case-1-q-owl} and \ref{fact:q-in-R-epsilon}.
\begin{fact}
  $s = (\preceq_s, \hat{s})$ is an order with levels on $X$, and $s \in \mathcal{R}_\epsilon$.
\end{fact}

The following fact arises directly from the definition of $\hat{s}$ and $\preceq_s$.
\begin{fact}\label{fact:order-coincidence-s}
  The order $\preceq_s$ coincides with $\preceq_r$ on $V_X$. The same is true on $Q_{\mathrm{E}}$ or $X^{(n-1)} \backslash Q_{\mathrm{E}}$.
\end{fact}

The following fact can be shown in a similar way to Facts~\ref{fact:tau-0-q-minimum-in-Q} and \ref{fact:P-subgraph}.
\begin{fact}\label{fact:S-subgraph}
  $S$ is a subgraph of $K(\isubg{s}{\tau_0}, \tilde{\omega})$.
\end{fact}

We can show the following fact in the same way as Fact~\ref{fact:q-tau0-ge-tau}.
From the following fact, we have $K(\isubg{s}{\tau_0}, \tilde{\omega}) \subseteq K(\subg{s}{\tilde{\tau}}, \tilde{\omega})$.
\begin{fact}\label{fact:case-2-tau-q-omega-0}
  $\tau_0 \succ_s \tilde{\tau}$.
\end{fact}

We can prove the following in a similar way to Facts~\ref{fact:case-1-omega-0-omega-1-differnt-cc} and \ref{fact:case-1-tau-q-omega-0} using
Facts~\ref{fact:order-coincidence-s}, \ref{fact:S-subgraph}, and \ref{fact:case-2-tau-q-omega-0}.
\begin{fact}
  $K(\subg{s}{\tilde{\tau}}, \tilde{\omega})$ and $K(\subg{s}{\tilde{\tau}}, \omega_0)$ are different
  connected components in $\subg{s}{\tilde{\tau}}$ and $\tilde{\tau}$ connects the two connected components.
  Therefore, $\tau_{s, \omega_0} = \tilde{\tau}$.
\end{fact}

From the above facts, we have the following.
\begin{fact}
  $\tilde{\omega} \not \in K(\subg{s}{\tau_{s, \omega_0}}, \omega_0)$.
\end{fact}

Finally, since $\OV(q, \omega_0) = \Kv(\subg{q}{\tau_{q, \omega_0}}, \omega_0)$, we conclude that $\tilde{\omega} \not \in \OV(q, \omega_0)$.

\subsection{Limitation of stable volumes}

In this subsection, we discuss some of the limitations of stable volumes.

The first relates to the $(r, \omega_0)$-order condition.
Theorem~\ref{thm:stable-tree} requires the $(r, \omega_0)$-order condition; however, this condition seems somewhat curious. To clarify, we offer a brief discussion of the role
of the condition.

First, we examine what happens when orders with levels $q$ and $r$ satisfy $\|\hat{q} - \hat{r}\| \leq \epsilon/2$ but not the $(r, \omega_0)$-order condition.
Figure~\ref{fig:r-omega-cond} shows an example. In this example, we assume the following conditions:
\begin{align}
  Y &= Y_0 \cup \{ \tau_1, \tau_2, \omega_1, \omega_2 \} \text{ is a simplicial complex}, \\
  \hat{r}(\sigma) &\approx \hat{q}(\sigma) \text{ for any } \sigma \in Y, \label{eq:r-oemga-approx-1} \\
  \hat{r}(\omega_1) &\approx \hat{r}(\omega_2),  \hat{q}(\omega_1) \approx \hat{q}(\omega_2), \label{eq:r-oemga-approx-2} \\
  \hat{r}(\sigma) &< \hat{r}(\tau_1)  \text{ for any } \sigma \in Y_0, \\
  \hat{q}(\sigma) &< \hat{q}(\tau_1)  \text{ for any } \sigma \in Y_0, \\
  \hat{r}(\tau_1) &< \hat{r}(\tau_2) < \hat{r}(\omega_1) < \hat{r}(\omega_2) \text{ for any } \sigma \in Y_0, \\
  \hat{q}(\tau_1) &< \hat{q}(\tau_2) < \hat{q}(\omega_2) < \hat{q}(\omega_1) \text{ for any } \sigma \in Y_0 .
\end{align}

\begin{figure}[htbp]
  \centering
  \includegraphics[width=\hsize]{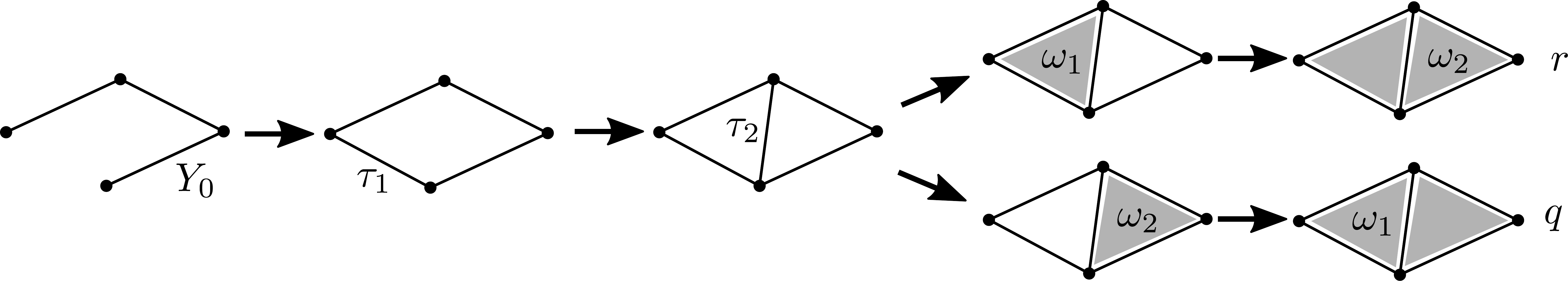}
  \caption{Two filtrations on the same simplicial complex}
  \label{fig:r-omega-cond}
\end{figure}

Here, we have
\begin{align*}
  (\hat{r}(\tau_1), \hat{r}(\omega_2)), (\hat{r}(\tau_2), \hat{r}(\omega_1)) &\in \pd_1(X, r), \\
  (\hat{q}(\tau_1), \hat{q}(\omega_1)), (\hat{q}(\tau_2), \hat{q}(\omega_2)) &\in \pd_1(X, q), \\
  (\hat{r}(\tau_1), \hat{r}(\omega_2)) & \approx (\hat{q}(\tau_1), \hat{q}(\omega_1)), \\
  (\hat{r}(\tau_2), \hat{r}(\omega_1)) & \approx (\hat{q}(\tau_2), \hat{q}(\omega_2)). 
\end{align*}
The optimal volumes are
\begin{align}
  \{\omega_1, \omega_2 \} &= \OV(\hat{r}, \hat{r}(\tau_1), \hat{r}(\omega_2)) = \OV(\hat{q}, \hat{q}(\tau_1), \hat{q}(\omega_1)),
                            \label{eq:r-omege-large-ov}\\
  \{\omega_1 \} & = \OV(\hat{r}, \hat{r}(\tau_2), \hat{r}(\omega_1)), \\
  \{\omega_2 \} &=  \OV(\hat{q}, \hat{q}(\tau_2), \hat{q}(\omega_2)). 
\end{align}
Therefore, $(\hat{r}(\tau_2), \hat{r}(\omega_1)) \approx (\hat{q}(\tau_2), \hat{q}(\omega_2))$; however,
$\OV(\hat{r}, \hat{r}(\tau_2), \hat{r}(\omega_1))$ does not intersect with $\OV(\hat{q}, \hat{q}(\tau_2), \hat{q}(\omega_2))$.
Thus, the example shows that $\OV(\hat{r}, \hat{r}(\tau_2), \hat{r}(\omega_1))$ does not have a robust part to small noise
if we do not assume the $(r, \omega_0)$-order condition.

In contrast, \eqref{eq:r-omege-large-ov} indicates that $\OV(\hat{r}, \hat{r}(\tau_1), \hat{r}(\omega_2))$ has a stable part 
even if we do not assume the $(r, \omega_0)$-order condition. We expect that if an optimal volume is large in some sense,
then the optimal volume has a stable part even if the $(r, \omega_0)$-order condition is not satisfied.
Mathematically clarifying what is meant by ``large in some sense'' is a subject for future research.

Next, we focus on the construction of simplicial filtrations. An alpha complex~\cite{alpha-shape} is often used to construct a filtration
from a pointcloud. Alpha complexes have some good properties. If a pointcloud is in $\R^d$ and satisfies the general position condition,
then the alpha shape is embedded in $\R^d$. An alpha shape has the same information as the union-balls model. More precisely, an alpha shape
is homotopy equivalent to the union of balls. However, the alpha shape may vary discontinuously with noise. Theorem~\ref{thm:stable-tree} assumes that
the simplicial complex $X$ does not change with noise; however, the assumption does not hold if the noise is large.
On the other hand, if the noise is small, the assumption holds and the theorem works perfectly. In Section~\ref{sec:example}, we
show that a large noise bandwidth parameter gives unexpected results. Therefore, it is reasonable to assume that
the noise is small; however, care should be taken.

\section{Another formalization of a stable volume as the solution to an optimization problem}\label{sec:stable-volume-optimization}

The formalization of stable volumes is based on persistence trees. Importantly, however, persistence trees are available only for the $(n-1)$th persistence homology,
and we cannot apply the concept directly to another degree.
Another version of a stable volume can be defined as the solution to an optimization problem. 

\begin{definition}\label{def:sv-opt}
  Let $r$ be an order with levels and  $(\tau_0, \omega_0) \in \bdsp_k(\X_r)$.
  The \emph{stable volume by optimization} of $(\tau_0, \omega_0)$ is defined by the solution to the following minimization problem.
  \begin{align}
    \text{minimize }& \|z\|_0, \text{ s. t. } \nonumber \\
    z & = \omega_0 + \sum_{\omega \in \mathcal{F}_{\epsilon, k+1} } \alpha_\omega \omega \in C_{k+1}(X; \Bbbk), \label{eq:z-equal} \\
    \tau^*(\partial z) &= 0 \text{ for any } \tau \in \mathcal{F}_{\epsilon, k}, \label{eq:tau-d-z-zero}
  \end{align}
  where $\Bbbk$ is the coefficient field and
  $\mathcal{F}_{\epsilon, k} = \{\sigma \in X^{(k)} \mid \hat{r}(\tau_0) + \epsilon \leq \hat{r}(\sigma) \text{ and } \sigma \prec_r \omega_0 \}$.
  $\SVo_\epsilon(r, \tau_0, \omega_0)$ denotes the stable volume by optimization.
\end{definition}

\subsection{Another formalization of a stable volume: A special case}\label{subsec:stable-volume-optimization-n-1}

For $(n-1)$th PH, we will prove the following second main theorem, which guarantees the validity of the definition.

\begin{theorem}\label{thm:sv-opt-tree}
  If $X$ satisfies Condition~\ref{cond:rn}, $\Bbbk=\Z/2\Z$, and $(\tau_0, \omega_0) \in \bdsp_{n-1}(\X_r)$,
  $\SVo_\epsilon(r, \tau_0, \omega_0)$ coincides with $\SV_\epsilon(r, \tau_0, \omega_0)$.
\end{theorem}

Recall the remark associated with Theorem~\ref{thm:vochd-tree} that  $\SVo_\epsilon(r, \tau_0, \omega_0)$ can be regarded as a subset of $X^{(n)}$.

For $z = \omega_0 + \sum_{\omega \in \mathcal{F}_{\epsilon,n}} \alpha_\omega \omega \in C_n(X; \Z/2\Z)$, we define $V_z$ and $E_z$ as follows.
\begin{equation}
  \begin{aligned}
    V_z &= \{ \omega \in \mathcal{F}_{\epsilon, n} \mid \alpha_\omega = 1\} \cup \{\omega_0\}, \\
    E_z &= \{ \tau \in \mathcal{F}_{\epsilon, n-1} \mid \exists \omega \in V_z\ \tau^* (\partial \omega) = 1\}.
  \end{aligned}
\end{equation}

To prove the theorem, we show the following lemma.
\begin{lemma}\label{lemma:G-z}
  If $z$ satisfies \eqref{eq:z-equal} and \eqref{eq:tau-d-z-zero},
  $G_z := V_z \cup E_z$ is a subgraph of $G_X$ and $G_z$ is a disjoint union of
  some connected components of $G_X(\hat{r} \geq \hat{r}(\tau_0) + \epsilon)$.
\end{lemma}

We now assume that $z$ satisfies \eqref{eq:z-equal} and \eqref{eq:tau-d-z-zero}. To show the above lemma, we first show the following fact. 
\begin{fact}
  For any $\tau \in E_z$, $\tau$ is not a face of $\omega_\infty$.
\end{fact}
\begin{proof}
  Assume $\tau$ is a face of $\omega_\infty$. Then there is a unique $n$-simplex $\omega_1 \in X^{(n)}$ which is the coface of $\tau$.
  From the definition of $E_z$ and $\omega_\infty \not \in V_z$, $V_z$ should contain $\omega_1$. However,
  \begin{equation}
    \tau^*(\partial z) = \tau^*(\partial (\sum_{\omega \in V_z} \omega)) = \sum_{\omega \in V_z} \tau^*(\partial \omega)
    = \tau^*(\partial \omega_1) = 1 \not = 0,
  \end{equation}
  which contradicts \eqref{eq:tau-d-z-zero}.
\end{proof}

The following fact guarantees that $G_z$ is a subgraph of $G_X$.
\begin{fact}\label{fact:edge-to-vertex}
  For any $\tau \in E_z$,
  both cofaces of $\tau$, $\omega_1$ and $\omega_2$, are
  contained in $V_z$.
\end{fact}
\begin{proof}
  From the definition $E_z$, $V_z$ contains either $\omega_1$ or $\omega_2$ and one of 
  the following holds: $\alpha_{\omega_1}= 1$ or $\alpha_{\omega_1} = 1$.
  From the condition \eqref{eq:tau-d-z-zero},
  \begin{equation}
    \tau^*(\partial z) = \sum_{\omega \in \mathcal{F}_{\epsilon, n}} \alpha_\omega \tau^*(\partial( \omega))
    = \alpha_{\omega_1} + \alpha_{\omega_2} = 0.
  \end{equation}
  Therefore, $\alpha_{\omega_1} = \alpha_{\omega_2} = 1$ , which means that both $\omega_1$ and $\omega_2$ are contained in $V_z$.
\end{proof}

\begin{fact}\label{fact:vertex-to-edge}
  For $\omega_1 \in V_X$, $A_{\omega_1}$ denotes the set of all $(n-1)$-dimensional faces of $\omega_1$.
  If $\omega_1 \in V_z$, then the following holds:
  \begin{equation}
    A_{\omega_1} \cap E_X(\hat{r} \geq \hat{r}(\tau_0) + \epsilon) = A_{\omega_1} \cap E_z.
  \end{equation}
\end{fact}
\begin{proof}
  Since $E_z \subseteq \mathcal{F}_{\epsilon, n-1} \subseteq E_X(\hat{r} \geq \hat{r}(\tau_0) + \epsilon)$,
  $A_{\omega_1} \cap E_X(\hat{r} \geq \hat{r}(\tau_0) + \epsilon) \supseteq A_{\omega_1} \cap E_z$ is trivial. Therefore,
  we assume $\tau \in A_{\omega_1} \cap E_X(\hat{r} \geq \hat{r}(\tau_0) + \epsilon)$ and will show
  $\tau \in A_{\omega_1} \cap E_z$. Since $\tau \prec_r \omega_1$ and $\omega_1 \in V_z \subset \mathcal{F}_{\epsilon, n} \cup \{\omega_0\}$,
  we have $\tau \prec_r \omega_0$. Therefore, $\tau \in \mathcal{F}_{\epsilon, n-1}$ since $\tau \in E_X(\hat{r} \geq \hat{r}(\tau_0) + \epsilon)$.
  We also have $\tau^*(\partial \omega_1) = 1$ because $\tau \in A_{\omega_1}$ and we conclude $\tau \in E_z$.
\end{proof}

By repeatedly using Fact~\ref{fact:vertex-to-edge}, we can see that any path in $G_X(\hat{r} \geq \hat{r}(\tau_0) + \epsilon)$ starting from
$\omega_1 \in V_z$ is also a path in $G_z$. This means that $G_z$ is a disjoint union of
some connected components of $G_X(\hat{r} \geq \hat{r}(\tau_0) + \epsilon)$.

\begin{lemma}\label{lemma:cc-opt}
  $z_1 = \sum_{\omega \in \Kv(G_X(\hat{r} \geq \hat{r}(\tau_0) + \epsilon), \omega_0) } \omega$ satisfies \eqref{eq:z-equal} and \eqref{eq:tau-d-z-zero}.
\end{lemma}

\begin{proof}
  First, we show \eqref{eq:z-equal}. The condition \eqref{eq:z-equal} is equivalent to the following inclusion relationship:
  \begin{equation}\label{eq:cc-in-F}
    \Kv(G_X(\hat{r} \geq \hat{r}(\tau_0) + \epsilon), \omega_0) \subset \mathcal{F}_{\epsilon, n} \cup \{\omega_0\}.
  \end{equation}
  Since $G_X(\hat{r} \geq \hat{r}(\tau_0) + \epsilon)$ is a subgraph of $\subg{r}{\tau_0}$ and $\omega_0$ is
  $(\succeq_r)$-maximum element in $K(\subg{r}{\tau_0}, \omega_0)$, $\omega \preceq_r \omega_0$ holds for any
  $\omega \in \Kv(G_X(\hat{r} \geq \hat{r}(\tau_0) + \epsilon), \omega_0)$. This equates to \eqref{eq:cc-in-F}.

  Next we show \eqref{eq:tau-d-z-zero}. We can show the following relationship in a similar way to \eqref{eq:cc-in-F}.
  \begin{equation}
    K_E(G_X(\hat{r} \geq \hat{r}(\tau_0) + \epsilon), \omega_0) \subseteq \mathcal{F}_{\epsilon, n-1}.
  \end{equation}
  Let $\tau$ be an element of $\mathcal{F}_{\epsilon, n-1}$. We consider the following two
  cases to show $\tau^*(\partial z) = 0$.
  \begin{enumerate}
  \item $\tau \in K_E(G_X(\hat{r} \geq \hat{r}(\tau_0) + \epsilon), \omega_0)$.
    In this case, $K_E(G_X(\hat{r} \geq \hat{r}(\tau_0) + \epsilon), \omega_0)$ contains both cofaces $\omega_1, \omega_2$ of $\tau$, and
    \begin{equation}
      \tau^*(\partial z_1) = \tau^*(\partial \omega_1) + \tau^*(\partial \omega_2) = 1 + 1  = 0.
    \end{equation}
  \item $\tau \in \mathcal{F}_{\epsilon, n-1} \backslash K_E(G_X(\hat{r} \geq \hat{r}(\tau_0) + \epsilon), \omega_0)$.
    In this case, $K_E(G_X(\hat{r} \geq \hat{r}(\tau_0) + \epsilon), \omega_0)$ does not contain both cofaces of $\tau$, and
    $\tau^*(\partial z_1) = 0$.
  \end{enumerate}
  In both cases, we have $\tau^*(\partial z_1) = 0$.
\end{proof}

From Lemma~\ref{lemma:G-z} and Lemma~\ref{lemma:cc-opt}, we can show that $z_1$ in Lemma~\ref{lemma:cc-opt} is the solution to the optimization
problem in Definition~\ref{def:sv-opt}. This means that 
$\Kv(G_X(\hat{r} \geq \hat{r}(\tau_0) + \epsilon), \omega_0) = \SVo_{\epsilon}(\hat{r}, \tau_0, \omega_0)$.
From \eqref{eq:sv-by-kv}, this set is equal to $\SV(r, \tau_0, \omega_0)$ , which completes the proof of the theorem.

\subsection{Another formalization of a stable volume: general case}\label{subsec:stable-volume-optimization-general}

We can apply Definition~\ref{def:sv-opt} to cases other than the above. In such cases,
Theorem~\ref{thm:sv-opt-tree} does not hold in general, and it is difficult to mathematically ensure the good property
shown in Theorem \ref{thm:stable-tree}. Empirically, however, the stable volumes often work well.
(We examine the property using computer experiments in Section~\ref{sec:example}.)

The reason for this is likely that an optimal volume is often included in a lower-dimensional structure (a submanifold or a lower-dimensional
simplicial complex) and the solution of the stable volume is also included in the structure.

\subsection{Stable sub-volume}

In the previous subsection, we expressed our belief that a lower-dimensional structure may make the stable volume work well.
We can develop a variant of stable volume using this idea.
If we find such a lower-dimensional structure, we can consider the
stable volume constrained on the structure. An optimal volume is
one possible lower-dimensional structure. We define a \emph{stable sub-volume}
as follows.

\begin{definition}\label{def:ssv-opt}
  Let $r$ be an order with levels and  $(\tau_0, \omega_0) \in \bdsp_k(\X_r)$.
  The \emph{stable sub-volume} of $(\tau_0, \omega_0)$ is defined by the solution to the following minimization problem:
  \begin{align}
    \text{minimize }& \|z\|_0, \text{ s. t. } \nonumber \\
    z & = \omega_0 + \sum_{\omega \in \mathcal{F}_{\epsilon, k+1} \cap \OV(r, \tau_0, \omega_0)} \alpha_\omega \omega \in C_{k+1}(X; \Bbbk), \label{eq:ssv-z-equal} \\
    \tau^*(\partial z) &= 0 \text{ for any } \tau \in \mathcal{F}_{\epsilon, k}, \label{eq:ssv-tau-d-z-zero}
  \end{align}
  where $\Bbbk$ is the coefficient field and
  $\mathcal{F}_{\epsilon, k} = \{\sigma \in X^{(k)} \mid \hat{r}(\tau_0) + \epsilon \leq \hat{r}(\sigma) \text{ and } \sigma \prec_r \omega_0 \}$.
\end{definition}
For stable sub-volumes, the parameter $\epsilon$ is also called a noise
bandwidth parameter.

The following proposition holds for the $(n-1)$th PH in $\R^n$ since the stable volume is a subset of an optimal volume.
\begin{prop}
  Let $X$ be a simplicial complex embedded in $\R^n$ satisfying
  Condition~\ref{cond:rn}. Let $r$ be an order with levels.
  For $(\tau_0, \omega_0) \in \bdsp_{n-1}(\X_r)$ and $\epsilon > 0$,
  the stable volume of $(\tau_0, \omega_0)$ with bandwidth parameter $\epsilon$
  coincides with
  the stable sub-volume of $(\tau_0, \omega_0)$ with bandwidth parameter $\epsilon$.
\end{prop}

The advantages and disadvantages of a stable sub-volume will be discussed in Section~\ref{sec:comparison}.

\section{Implementation}\label{sec:implementation}

In this section, we discuss how to produce the stable volume using a computer.

While we can directly implement stable volumes by persistence trees 
(Definition \ref{def:sv-tree}) using the persistence trees constructed by
Algorithm~\ref{alg:volopt-hd-compute}, 
implementing stable volumes as an optimization problem is more difficult. As discussed in \cite{Chen2011,p-1-cycle}, these kinds of optimization
problems are NP-hard in general, which makes them difficult to solve on
a computer. To resolve this problem, we can apply the following techniques:
\begin{itemize}
\item Use $\R$ as a coefficient field instead of $\Z/2\Z$
\item Change $\ell^0$ norm to $\ell^1$ norm
\end{itemize}
The second technique is often used in sparse modeling. Following the above
approximations, the optimization problem can be formulated as a linear programming problem.

We first need to translate the problem into an acceptable form for a linear programming solver. 
That is, we need to specify the variables,
the objective function
(the function to be minimized), and the constraints in the following form:
\begin{itemize}
\item The objective function should be \\
  $(\mbox{a linear combination of variables}) + (\mbox{constant})$
\item Each constraint should be \\
  $(\mbox{a linear combination of variables}) + (\mbox{constant})\ \textrm{OP}\ (\mbox{another constant})$,
  where $\textrm{OP}$ is any of the relational operators $=, \geq, \leq$
\end{itemize}
We can translate this in the same way as in \cite{voc}.
\begin{align*}
  &\mbox{variables: }  \alpha_\omega, \bar{\alpha}_\omega \mbox{ for all } \omega \in \mathcal{F}_{\epsilon, k+1}\\
  &\mbox{minimize } \sum_{\omega \in \mathcal{F}_{\epsilon, k+1}} \bar{\alpha}_\omega, \mbox{ subject to}\\
  &\ \ \bar{\alpha}_\omega - \alpha_\omega \geq 0 \mbox{ for each } \omega \in \mathcal{F}_{\epsilon, k+1},\\
  &\ \ \bar{\alpha}_\omega + \alpha_\omega \geq 0 \mbox{ for each } \omega \in \mathcal{F}_{\epsilon, k+1}, \\
  &\ \ c_{\omega_0, \tau} + \sum_{\omega \in \mathcal{F}_{\epsilon, k+1} } c_{\omega, \tau} \ \alpha_\omega = 0 \mbox{ for each } \tau \in \mathcal{F}_{\epsilon, k}, \\
  &\text{where} \\
  & \ \ c_{\omega, \tau} = \tau^*(\partial \omega).
\end{align*}
It should be noted that $\R$ is used as the coefficient field and that we need to
consider the sign of $c_{\omega, \tau}$.

Section 4.2 in \cite{voc} introduced the idea of accelerating the computation of
optimal volume using locality. The same idea is applicable to
stable volumes. 

The program is already in HomCloud\cite{homcloud} and was used for the examples in Section~\ref{sec:example}.





\section{Examples}\label{sec:example}

In this section, we give examples of
stable volumes and stable sub-volumes.
Alpha filtration\cite{alpha-shape} is used to compute the persistence diagrams.

\subsection{3x3x3 lattice points}

For the example shown in Fig.~\ref{fig:crystalline},
we first produced 27 points in three-dimensional space on a lattice. Specifically,
the pointcloud here is $\{(0, 0, 0), (0, 0, 1), \ldots, (2, 2, 2)\}$. We added
a small noise to each point sampled from a uniform distribution on $(-0.05, 0.05)^3$, and
a persistence diagram was computed from the pointcloud (Fig.~\ref{fig:crystalline}(b)).
The diagram includes 28 birth-death pairs near $(1/4, 1/\sqrt{2}) \approx (0.5, 0.7)$,
each corresponding to 1x1 squares.

Figure~\ref{fig:crystalline}(c) shows two optimal cycles;
; Fig.~\ref{fig:crystalline}(d) shows two other optimal cycles.
Figure~\ref{fig:crystalline-2}(b) shows two stable volumes of the same pairs
as in Fig.~\ref{fig:crystalline}(d)
with a bandwidth parameter $\epsilon = 0.05$.

Table~\ref{tab:lattice3x3x3_stat} shows the numbers of volumes having different numbers of vertices among 
all optimal volumes and stable volumes . As indicated, twenty of the optimal volumes are squares, while eight are not.
In contrast, all the stable volumes are squares.

\begin{table}[hbtp]
  \centering
  \caption{Number of volumes having a number of vertices among all optimal volumes and stable volumes}
  \label{tab:lattice3x3x3_stat}

  \begin{tabular}{c|ccc} 
    & 4 vertices & 6 vertices & 10 vertices \\ \hline
    number of optimal volumes & 20 & 7 & 1 \\
    number of stable volumes & 28 & 0 & 0 \\
  \end{tabular}
\end{table}

We also computed stable sub-volumes in this setting and confirmed that
the stable sub-volumes coincided with the stable volumes.

\subsection{2D lattice with random defects}\label{subsec:lattice-2d}

For the example shown in Fig.~\ref{fig:lattice-2d},
we prepared 30x30 lattice points in a 2D space. The distance between
vertices is 1. The points were randomly removed with a probability of 0.5. The
points on the perimeter were not removed. We added
a small noise to each point sampled from a uniform distribution on $(-0.1, 0.1)^2$.
A 1st PD was computed from the pointcloud.

Figure~\ref{fig:lattice-2d}(c) shows the optimal volume of (0.496, 4.371). Figure~\ref{fig:lattice-2d}(d) shows the stable volume of the same birth-death pair
with bandwidth parameter $\epsilon = 0.12$.

We also examined the effect of changing the bandwidth parameter,
computing stable volumes for bandwidth parameter
$\epsilon = 0.0, 0.01, 0.02, \ldots, 0.40$.
Figure~\ref{fig:lattice2d-stat} shows the graph of the bandwidth parameter
versus the size of the stable volumes. Size was measured as the number of
simplices in the stable volume.

\begin{figure}[htbp]
  \centering
  \includegraphics[width=0.7\hsize]{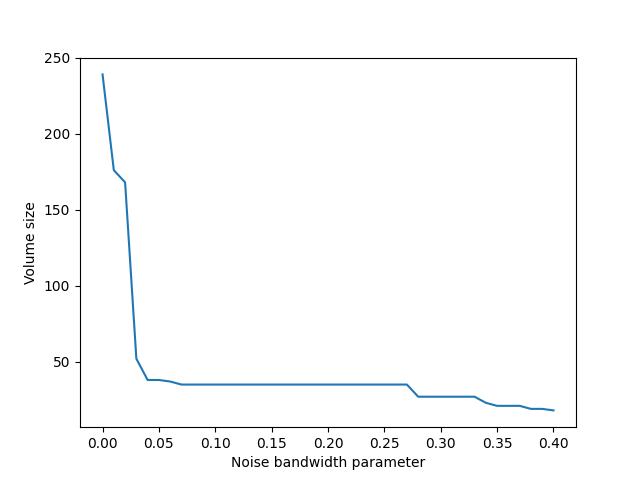}
  \caption{Bandwidth parameter versus size of stable volume}
  \label{fig:lattice2d-stat}
\end{figure}

From the optimal volume ($\epsilon = 0$), the size of the stable volume
rapidly decreases. A wide plateau appears at $\epsilon = 0.04$ and
is completely flat from $\epsilon = 0.06 $ to $\epsilon = 0.27$,
meaning that the stable volume is stable to changes in
bandwidth parameter over the range $[0.06, 0.27]$. This suggests that
$\epsilon$ should be somewhere within this range.
It should be noted that the scale of this range coincides with the scale of the noise.
This is consistent with the fact that
the bandwidth parameter indicates the strength of the virtual noise.

\subsection{Comparison with previous studies}

We apply the statistical method and reconstructed shortest cycles shortest using the same data as in Figure~\ref{fig:lattice-2d} to compare with our method.

Figure~\ref{fig:lattice-2d} shows the results of statistical method.
We use Algorithm~\ref{alg:stat} to compute these results with a random uniform 
perturbation to each point. The strenght of the noise is $[-0.03,0.03]^2$ and
$[-0.06,0.06]^2$. We put large black dots in the figure to indicate points
whose frequencies are more than 70\% or 90\%. The result is consistent with
Fig~\ref{fig:lattice-2d} and has richer information. However the result
looks more difficult to interpret than Fig~\ref{fig:lattice-2d}.

\begin{figure}[htbp]
  \centering
  \includegraphics[width=\hsize]{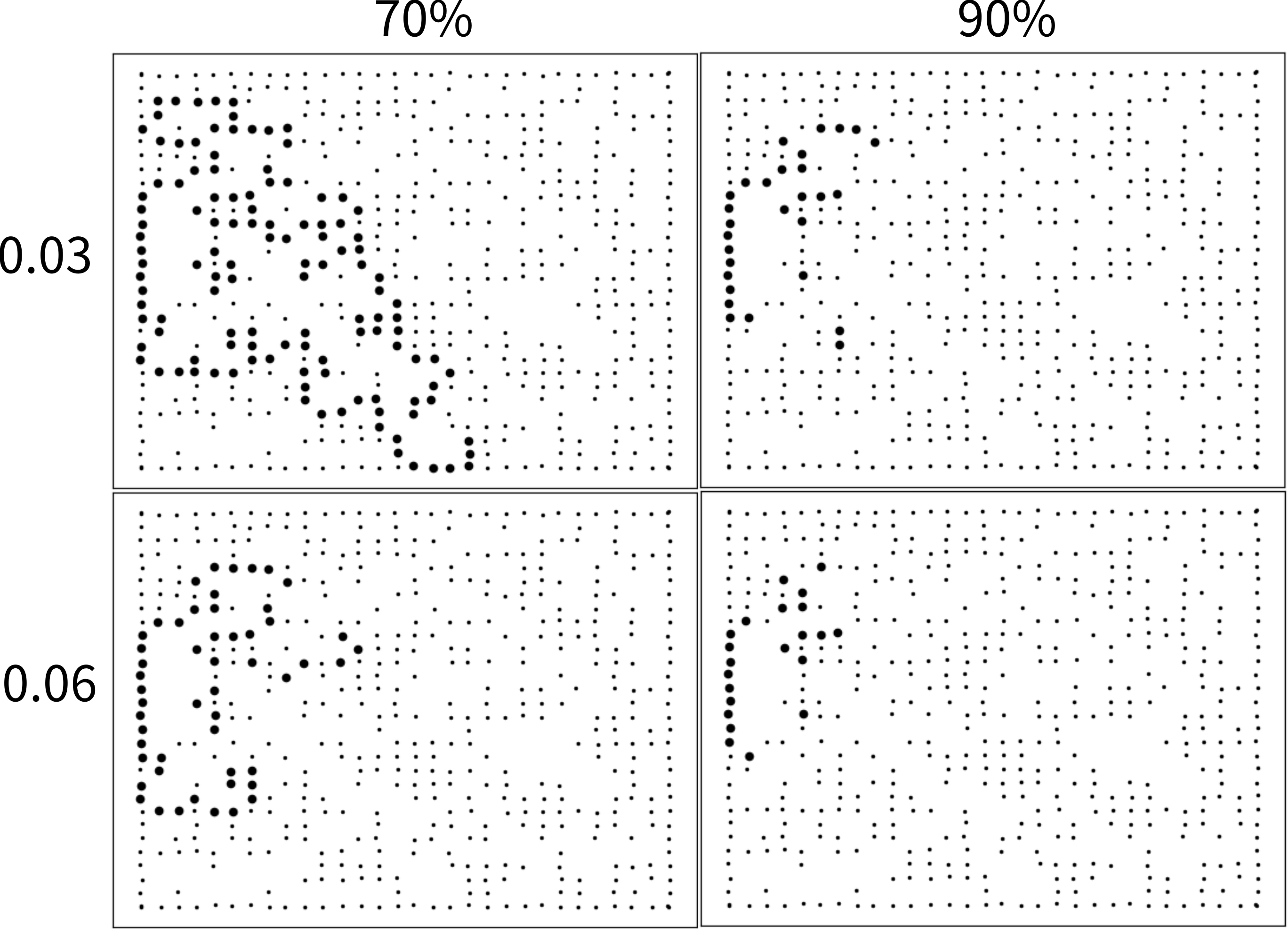}
  \caption{Result of the statistical method. Upper two figures: Additive noise from uniform distribution on $[-0.03, 0.03]^2$. Lower two figures: Additive noise from uniform distribution on $[-0.06, 0.06]^2$. The left two figures: The points with frequencies > 70\% is marked. The right two figures: The points with frequencies > 90\% is marked.}
  \label{fig:lattice2d-stat}
\end{figure}

Figure~\ref{fig:lattice2d-rsc} shows the reconstructed shortest cycles with
different noise bandwidth parameters, 0.1 (left) and 0.3 (right).
As shown in the figure, the result for 0.3 is consistent with other results but the result for 0.1 is not consistent.
This means that the shortest loop criterion sometimes gives a result inconsistent with the structure of persistent homology.

\begin{figure}[htbp]
  \centering
  \includegraphics[width=\hsize]{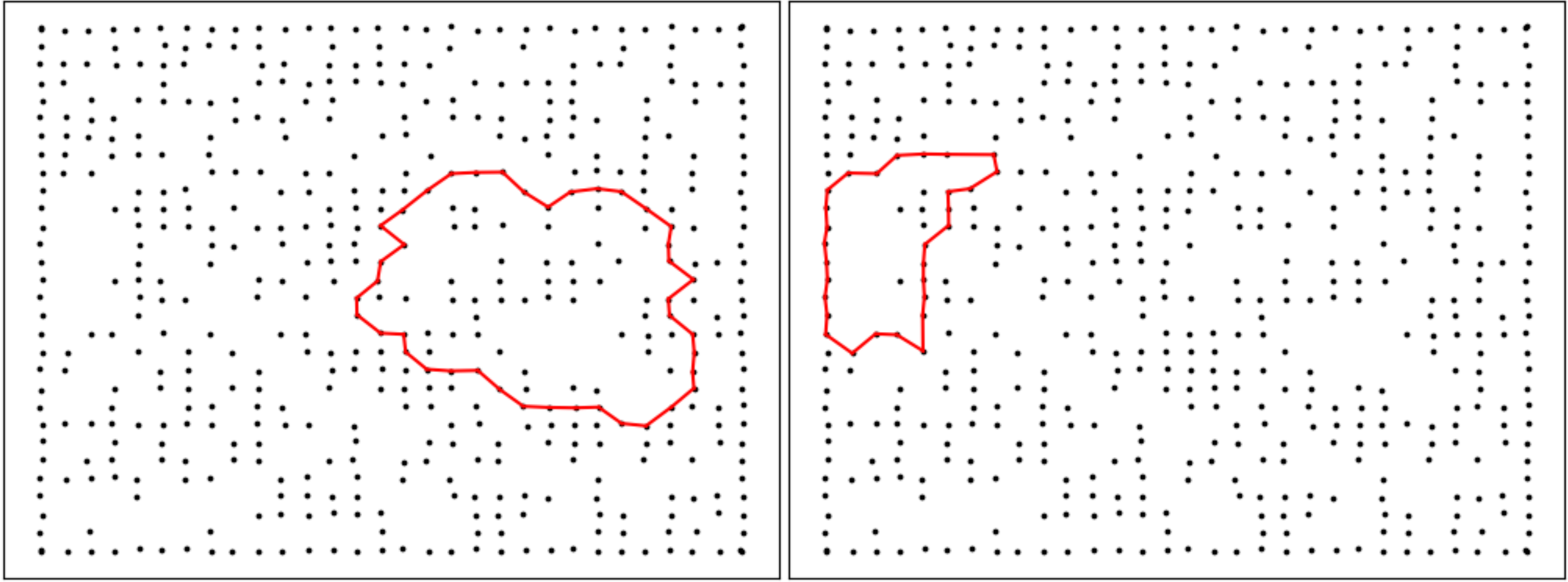}
  \caption{Result of reconstructed shortest cycles. Left: Noise bandwidth parameter is 0.1. Right: Noise bandwidth parameter is 0.3. }
  \label{fig:lattice2d-rsc}
\end{figure}

In this comparison, we also remark that the meaning of noise bandwidth parameters
used in those three methods are different, so direct comparison of the results with the same parameter is meaningless. We should focus on the changes of the results by the change of noise bandwidth parameters.


\subsection{Atomic configuration of amorphous silica}\label{subsec:sio2}

The proposed approach was applied to more realistic data. In this case, we used the atomic configuration of amorphous silica.
The data are from ISAACS\cite{isaacs},
generated by reverse Monte Carlo simulations guided by synchrotron X-ray radiation data.
The data are available at \url{http://isaacs.sourceforge.net/ex.html}.

The PDs were computed from the atomic configuration. Two types of atoms, silicons and oxygens, were mixed in a
1:2 ratio in the data. The atomic type was ignored in this example, and only the positions of the atoms were used. Figure~\ref{fig:sio2-pd} shows the 1st PD.

The PD has birth-death pairs on the vertical line with $(\text{birth time}) \approx 0.7$.
The birth-death pairs on the vertical line correspond to rings formed by the chemical bonds between oxygen and silicon. Oxygen and silicon atoms appear alternatively on these rings.
Previous studies~\cite{Hiraoka28062016, YoheiONODERA201919143} have found that 
the existence of the vertical line on a PD implies network structures.

\begin{figure}[hbtp]
  \centering
  \includegraphics[width=0.5\hsize]{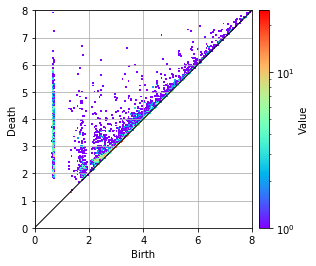}
  \caption{1st PD for amorphous silica}
  \label{fig:sio2-pd}
\end{figure}

\begin{figure}[hbtp]
  \centering
  \includegraphics[width=\hsize]{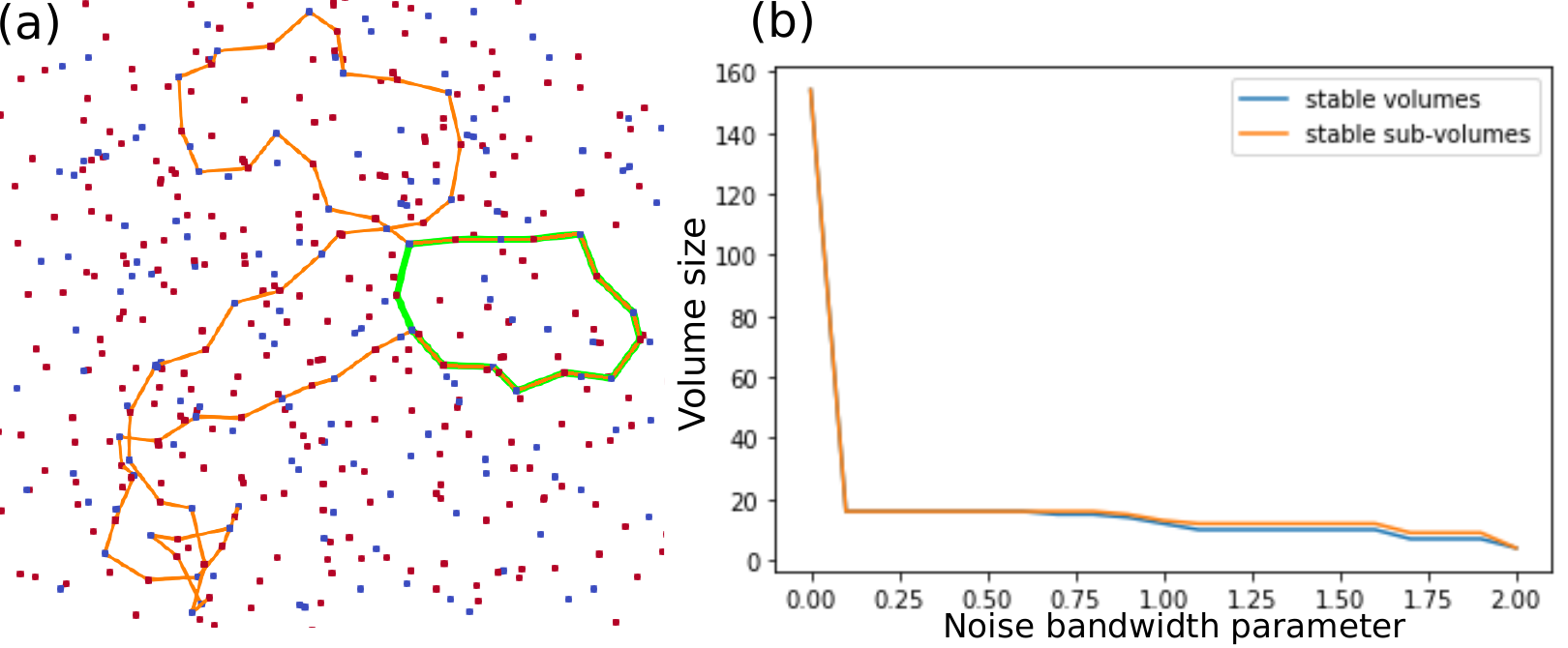}
  \caption{
    (a) Optimal volume and stable volume of (0.677, 5.007). The orange ring is the boundary of the optimal volume; the green ring is the boundary of the stable volume. The red and blue points are oxygen atoms and silicon atoms, respectively.
    (b) The plot of the bandwidth parameter $\epsilon$ versus the size of stable volumes and stable sub-volumes.
  }
  \label{fig:sio2-1}
\end{figure}

Figure~\ref{fig:sio2-1}(a) shows the boundary of the 
optimal volume (orange) and the boundary of the stable volume (green) with $\epsilon = 0.2$ of (0.677, 5.007).
The red and blue points are oxygen atoms and silicon atoms, respectively. In this case, the stable volume and stable sub-volume are identical.
The optimal volume is a large twisted ring, while the stable volume is a simpler ring. Figure~\ref{fig:sio2-1}(b) shows the
$\epsilon$ versus volume plot. The plot indicates that the size of the stable volume quickly decreases from $\epsilon=0$ to
$\epsilon=0.1$, meaning that the optimal volume is sensitive to noise. Therefore, we can consider that
the stable volume is the essential part for the birth-death pair.

It should be noted that the stable volume with $\epsilon \in [1.1, 1.6]$ on the second plateau is not a reasonable representation of the birth-death pair
since only oxygen atoms exist on the boundary of the volume. The large $\epsilon$ causes the removal of the information of the silicon atoms.

Figure~\ref{fig:sio2-2}(a) shows the optimal volume (orange) of (0.669, 5.053)
and the stable volume (green) of the same pair with $\epsilon=0.2$. In this case, the optimal volume and the stable sub-volume
are identical. However, the stable volume and sub-volume are not identical since
the stable volume is not included in the optimal volume. Figure~\ref{fig:sio2-2}(b) shows the plot of $\epsilon$ against the volume.

\begin{figure}[hbtp]
  \centering
  \includegraphics[width=\hsize]{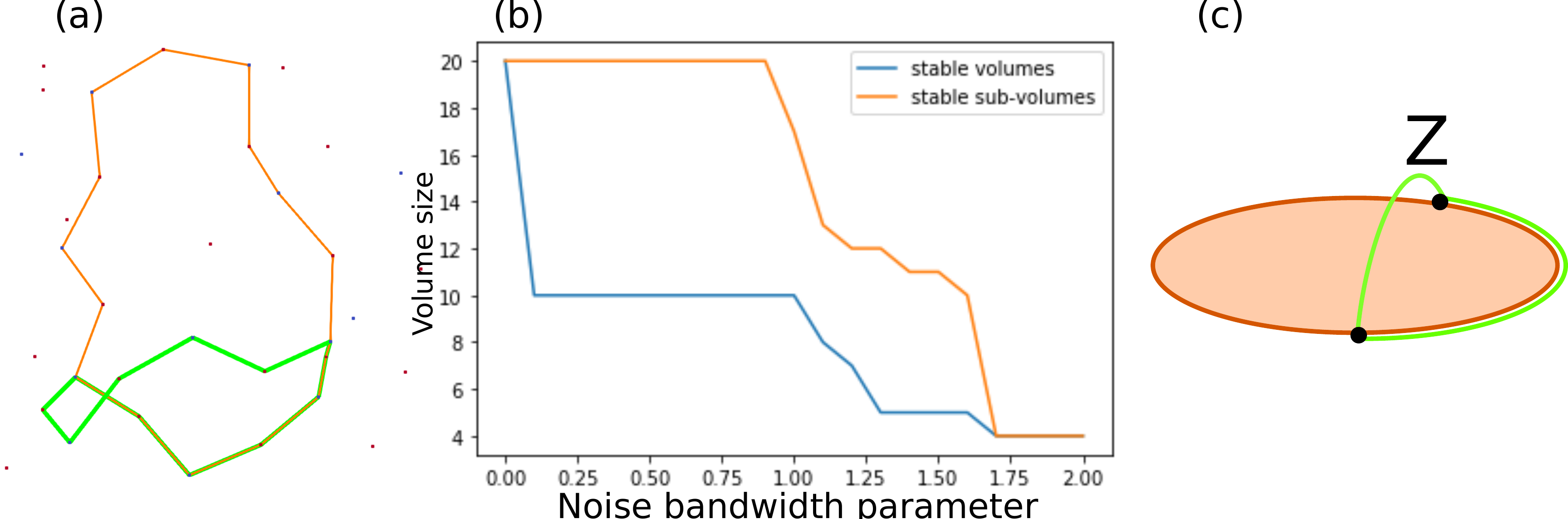}
  \caption{
    (a) Optimal volume and stable volume of (0.669, 5.053). The orange ring is the boundary of the optimal volume;
    the green ring is the boundary of the stable volume. The red and blue points are oxygen atoms and silicon atoms, respectively.
    (b) Plot of the bandwidth parameter $\epsilon$ versus the size of the stable volumes and stable sub-volumes.
    (c) Schematic of the optimal volume and the stable volume. The orange area is the optimal volume,
    the dark orange ring is its boundary, and the green ring is the boundary of the stable volume.
  }
  \label{fig:sio2-2}
\end{figure}

In Fig.~\ref{fig:sio2-2}(a), the stable volume is tighter than the stable sub-volume.
Figure~\ref{fig:sio2-2}(c) shows the schematic of the optimal volume and the stable volume.
The stable volume and sub-volume are different since path $Z$ in the figure is not included in the optimal volume. 
In our opinion, the stable volume appears superior since it surrounds the tunnel in the pointcloud more tightly.
However, we do not have a theoretical guarantee. In general, the stable volume is tighter than the stable sub-volume since
the optimization for a stable volume is more aggressive than that for a stable sub-volume.

\begin{figure}[hbtp]
  \centering
  \includegraphics[width=\hsize]{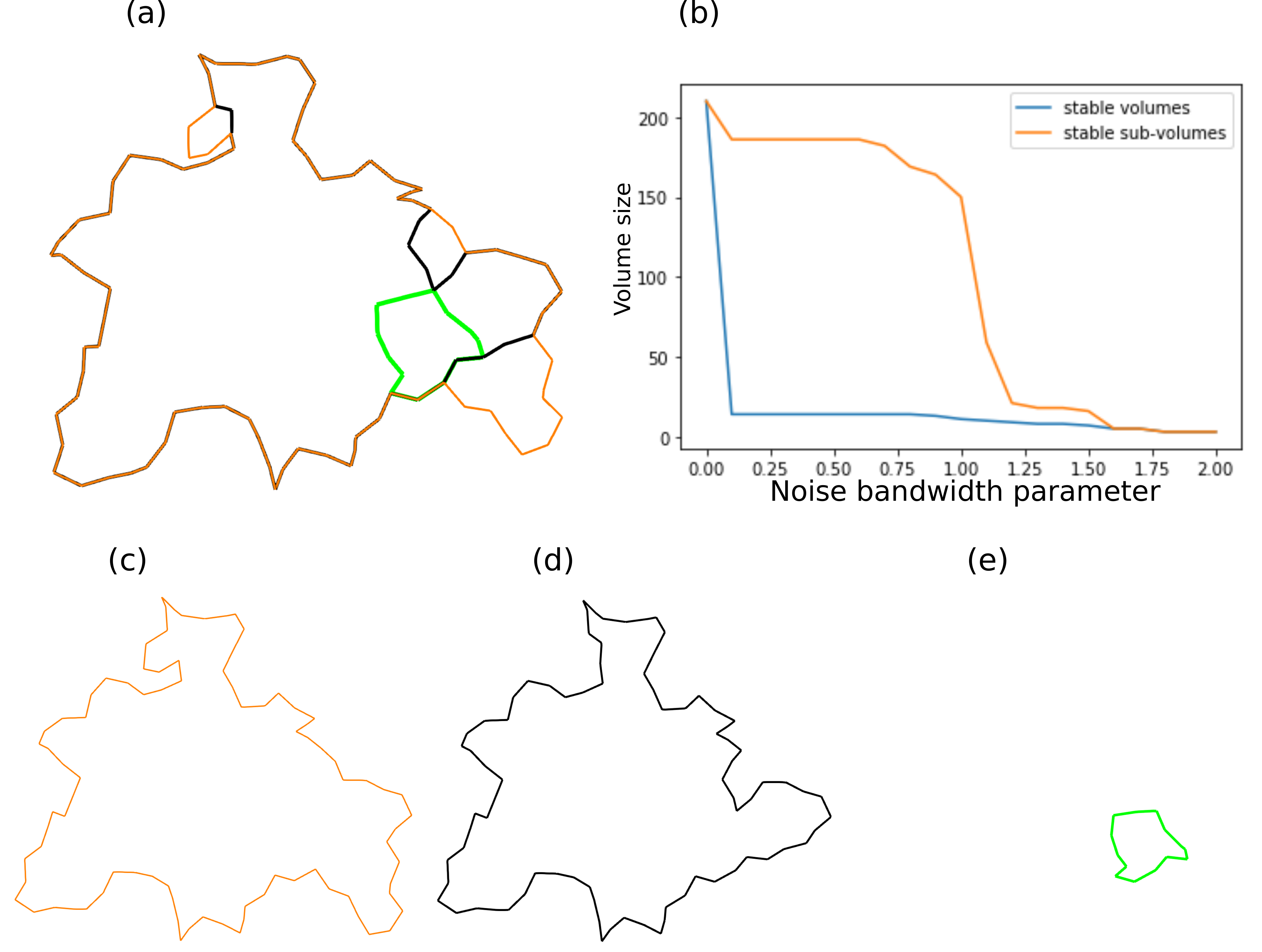}
  \caption{
    (a) Optimal volume and stable volume of (0.671, 5.063). The orange ring is the boundary of the optimal volume, 
    the green ring is the boundary of the stable volume, and the black ring is the boundary of the stable sub-volume.
    The atoms are omitted for the sake of clarity.
    (b) Plot of the bandwidth parameter $\epsilon$ versus the size of stable volumes and stable sub-volumes.
    (c, d, e) Optimal volume, stable sub-volume, and stable volume.
  } 
  \label{fig:sio2-4}
\end{figure}

Figure~\ref{fig:sio2-4}(a) shows the boundaries of the optimal volume (orange), stable volume (green), and stable sub-volume (black) of (0.671, 5.063).
The noise bandwidth parameter $\epsilon = 0.2$. In this example, the stable volume shows a very tight ring, while the stable sub-volume gives
a very conservative result.
The implication is that a stable volume gives a better result than a stable sub-volume when the optimal volume is extremely large.

\section{Tuning of noise bandwidth parameter}\label{sec:parameter-tuning}

In applying stable volumes, properly tuning the noise
bandwidth parameter $\epsilon$ is essential. The examples shown in the previous section suggest two possible approaches to such tuning:
\begin{itemize}
\item The scale of $\epsilon$ should be the same scale as
  the noise of the data
\item A value on the plateau shown in the plot of $\epsilon$ against
  volume size (Section~\ref{subsec:lattice-2d} and Fig.~\ref{fig:lattice2d-stat})
\end{itemize} should be used.

In applying the first approach, we can use the domain knowledge
about the data to tune the parameter.
Here, the scale of the system noise gives an estimate of the parameter.
For example, the scale of thermal fluctuation can be used as the parameter
if the data consist of the atomic configuration, and thermal fluctuation is the dominant noise.

To apply the second approach, we need to plot the figure.
As noted in Section~\ref{subsec:sio2}, when there are multiple plateaus, we need additional criteria to determine
which is better.

If we use stable volumes by persistence trees,
such a plot is easy to produce. From Definition~\ref{def:sv-tree}, the size
of a stable volume is computable as follows:
\begin{equation}\label{eq:stablevolume-size}
  1 +  \sum_{\omega \in C_\epsilon(\tau_0, \omega_0)} (\text{size of }\mathrm{dec}(\omega, V, E)).
\end{equation}
We can compute the size of $\mathrm{dec}(\omega, V, E)$ by counting
all the descendant elements of $\omega$. It is also easy to judge whether
$\omega \in C_\epsilon(\tau_0, \omega_0)$ by comparing
$\hat{r}(\tau_0)$ and $\hat{r}(\tau)$ in \eqref{eq:c-epsilon}.
HomCloud already has this functionality.

When we cannot use stable volumes by persistence trees and we want to plot $\epsilon$ against volume size,
a stable sub-volume is more useful than a stable volume.

\section{Comparison between stable volumes and stable sub-volumes}\label{sec:comparison}

The question of whether stable volumes or stable sub-volumes are better requires further discussion.
The differences between stable volumes and stable sub-volumes can be described as follows:
\begin{itemize}
\item If we want to compute stable volumes or stable sub-volumes using multiple bandwidth parameters,
  the computation cost of stable sub-volumes is smaller than that of stable volumes.
  Thus, stable sub-volumes are desirable for constructing $\epsilon$ versus volume plots.
\item If the size of the optimal volume is large, a stable volume often gives a more aggressive and 
  likely better result than a stable sub-volume.
\item A stable sub-volume gives a more conservative result than a stable volume.
\end{itemize}
The first listed characteristic shows the advantage of using a stable sub-volume, while the second and third characteristics imply that
the choice between stable volumes and stable sub-volumes depends on the problem. 

In the author's opinion, stable sub-volumes are easy to handle in general. However,
users should compare both options and make a decision as to which is better for their data.

\section{Concluding remarks}\label{sec:conclusion}

We have proposed the concept of stable volumes and stable sub-volumes
as a means for identifying good geometric realization of homology generators that, unlike many of the methods proposed in prior research, is robust to
noise.
The idea of stable volumes and stable sub-volumes is based on optimal volumes~\cite{voc}.

The statistical approach taken by 
Bendich, Bubenik, and Wagner~\cite{stabilization} offers another solution to
our problem. However, one advantage of stable volumes and stable sub-volumes is
that they do not require a large number of repeated computations.
Moreover, stable volumes and sub-volumes are easier to visualize than the output
of the statistical method
since they give a deterministic rather than a probabilistic description.
On the other hand, the advantage
of the statistical approach is its flexibility.
For first PH,
we can only apply the statistical approach when we want to minimize the length of a loop rather than minimize the volume. The statistical approach is also applicable to
the spatial distribution of points.
Another advantage of
the statistical approach is that it gives richer probabilistic information than stable volumes and sub-volumes.

Reconstruct shortest cycles in Ripserer.jl by Čufar\cite{cufar2020ripserer} offers the other solution.
The approach uses the representative of persistent cohomology and the shortest path algorithm.
Our proposed method has two advantages over reconstructed shortest cycles. The first is its mathematical background.
We prove Theorem~\ref{thm:stable-tree} to show the good property of stable volumes and stable sub-volumes.
Reconstruct shortest cycles do not have such mathematical justification.
The second is the difference in the scope of application.
Stable volumes and sub-volumes can apply to the PH of any degree; however, reconstructed shortest cycles are only applicable to 1st PH.
On the other hand, the advantage of reconstructed shortest cycles is computation efficiency.
The algorithm of reconstructed shortest cycles uses the shortest path algorithm and is faster than stable volumes and sub-volumes in general.

In this paper, we presented algorithms only for the filtration of
a simplicial complex;
however, the concept can be applied to other filtrations, such as cubical and
cell filtrations. The proposed algorithms should be helpful in the study of 
two-dimensional, three-dimensional, or higher-dimensional digital images.

Although we used the number of simplices as the volume size in this paper, it is possible to use the real volume by changing the objective function from $\sum_{\omega \in \mathcal{F}_{\epsilon, k+1}} \bar{\alpha}_\omega$ to $\sum_{\omega \in \mathcal{F}_{\epsilon, k+1}} \mathrm{vol}(\omega) \bar{\alpha}_\omega$, where $\mathrm{vol}(\omega)$ is the volume of the simplex $\omega$.
This idea may improve the result. 

\section*{Acknowledgments}
This research is partially
supported by JSPS KAKENHI JP19H00834,
JP20H05884, JST Presto JPMJPR1923, JST CREST JPMJCR15D3,
and JST-Mirai Program JPMJMI18G3.

\bibliographystyle{plain}
\bibliography{references}



\appendix
\section{Reconstructed shortest cycles in Ripserer.jl}\label{sec:rsc}

We illustrate the mechanism of reconstructed shortest cycles in Ripserer.jl using an example.

\begin{figure}[htbp]
  \centering
  \includegraphics[width=\hsize]{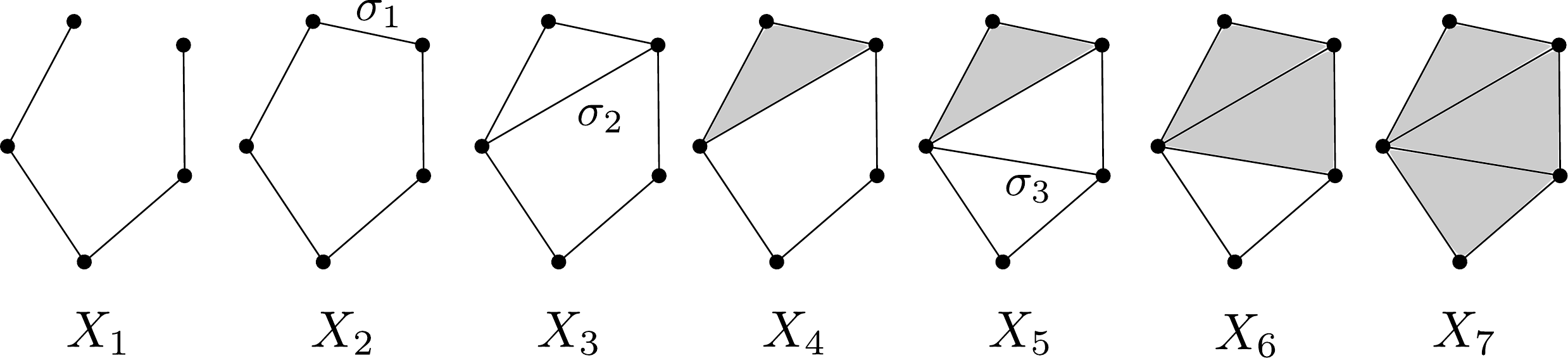}
  \caption{Example for reconstructed shortest cycles}\label{fig:rsc-example}
\end{figure}

Figure~\ref{fig:rsc-example} shows the example filtration, and we consider $\Z/2\Z$-persistent homology.
In this filtration, a homology generator appearing at $X_2$ disappears at $X_7$.
That is, $(2, 7)$ is a birth-death pair in the first persistence diagram. The representative cycle corresponding to the pair
is $\sum_{\sigma \in X_2^{(1)}} \sigma$. This representative cycle is the loop appearing at $X_2$.

Now we consider persistent cohomology. The persistence diagram defined by persistent cohomology is identical to the diagram by persistent homology since we use a field as a coefficient, but the representative \emph{cocycle} is, of course, different.
The representative cocycle corresponding to $(2, 7)$ is $\sigma_1^* + \sigma_2^* + \sigma_3^*$.

\begin{figure}[htbp]
  \centering
  \includegraphics[width=0.6\hsize]{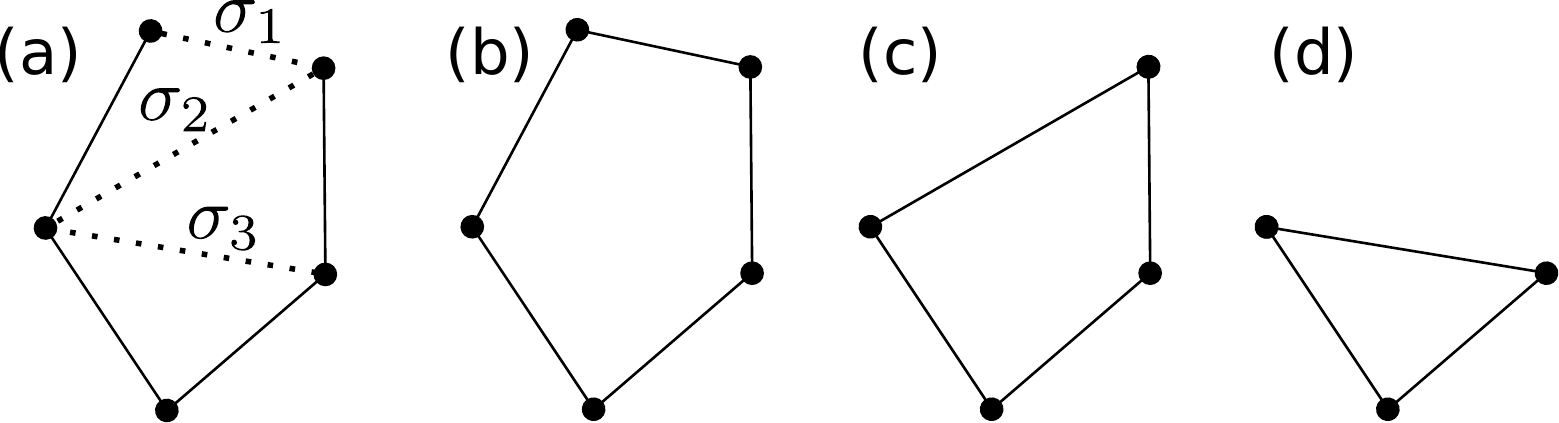}
  \caption{Representative cocycle}
  \label{fig:rsc-cocycle}
\end{figure}

From the observation of the example, we find that the representative cocycle is the ``cut'' of the loop.
This means that any representative cycle corresponding to the birth-death pair disappears if we remove all 1-simplices in the representative cocycle from
the 1-skeleton of $X_7$ as shown in Fig.~\ref{fig:rsc-cocycle}(a).

Using this fact, we can compute representative cycles using the shortest path algorithm.
Let $C$ be $\{\sigma_1, \sigma_2, \sigma_3\}$, the set of all 1-simplices in the representative cocycle.
Now $X_2^{(1)}$ is divided into two parts, one is $X_2^{(1)} \backslash C$ and another is $X_2^{(1)} \cap C = \{\sigma_1\}$.
Then we compute a shortest path in $X_4^{(1)} \backslash C$ whose two endpoints are the endpoints of $\sigma_1$.
The result is shown in Fig.~\ref{fig:rsc-cocycle}(b).

We can extend the idea to computer tighter cycles.
We can similarly divide $X_4^{(1)}$ into $X_4^{(1)} \backslash C$ and $X_4 ^{(1)} \cap C = \{\sigma_1, \sigma_2\}$.
For each $\sigma \in X_4^{(1)} \cap C$, we compute the shortest path in $X_4^{(1)} \backslash C$ whose two endpoints are the endpoints of $\sigma$.
The results are (b) and (c) in Fig.~\ref{fig:rsc-cocycle}, and we choose (c) as the shortest loop in these loops. Using $X_5$ instead of $X_4$,
we compute an even tighter loop (d). This is the mechanism of the reconstructed shortest cycles.

The algorithm is summarized as follows.
\begin{enumerate}
\item The 1st Persistence diagram $\pd_1(\X)$ for a filtration $\X: X_1 \subset \cdots \subset X_N$ is
  computed using persistent cohomology. The representative cocycles are also computed
\item We choose a birth-death pair $(b, d)$, and take the corresponding representative cocycle $\sum_{\sigma \in C} \sigma^*$
\item We also choose $k$ between $b$ and $d$
\item For each $\sigma \in X_k^{(1)} \cap C$, we compute the shortest path in $X_k^{(1)} \backslash C$, $P(\sigma)$, whose two endpoints are
  the endpoints of $\sigma$
\item We search the shortest one from $\{P(\sigma)\}_{\sigma \in C}$
\end{enumerate}
The mathematical justification of the algorithm is not yet given, but empirically it works well.
A deeper analysis of the algorithm is beyond the scope of this paper.

\end{document}